\documentclass[reqno]{amsart}
\usepackage[margin=1.15in]{geometry}
\usepackage{hyperref}
\usepackage{cite}
\usepackage{enumitem}

\theoremstyle{plain}
\newtheorem{theorem}{Theorem}
\newtheorem{proposition}[theorem]{Proposition}
\newtheorem{corollary}[theorem]{Corollary}
\newtheorem{lemma}[theorem]{Lemma}

\theoremstyle{definition}
\newtheorem{definition}{Definition}

\theoremstyle{remark}
\newtheorem{remark}{Remark}

\numberwithin{equation}{section}
\allowdisplaybreaks

\begin{document}

\author{Nhat Vy Huynh}
\address{Nhat Vy Huynh\newline
	Faculty of Mathematics and Computer Science, University of Science, Ho Chi Minh City, Vietnam;\newline
	Vietnam National University, Ho Chi Minh City, Vietnam;\newline
	Institute of Information Technology and Electrical, Electronics Engineering, University of Transport Ho Chi Minh City, Ho Chi Minh City, Vietnam}
\email{vy.huynh@ut.edu.vn}

\author{Phuong Le}
\address{Phuong Le (corresponding author, ORCID: 0000-0003-4724-7118)\newline
	Faculty of Economic Mathematics, University of Economics and Law, Ho Chi Minh City, Vietnam;\newline
	Vietnam National University, Ho Chi Minh City, Vietnam}
\email{phuongl@uel.edu.vn}


\title[Stability outside a compact set for anisotropic equations]{Solutions stable outside a compact set for anisotropic quasilinear equations with negative exponent or exponential nonlinearity}

\begin{abstract}
We study the equations $\Delta_p^H u=u^{-q}$, $u>0$, and $-\Delta_p^H u=e^u$ in $\mathbb{R}^N$, where $\Delta_p^H u=\operatorname{div}(H(\nabla u)^{p-1}\nabla H(\nabla u))$ is the Finsler $p$-Laplacian associated with a uniformly convex norm $H$, and $p\ge2$. We prove nonexistence of solutions stable outside a compact set for the negative-power equation when $p\le N<\frac{p(p+3)}{p-1}$ and $q$ lies above the sharp critical threshold, and for the exponential equation when $p<N<\frac{p(p+3)}{p-1}$. Consequently, neither equation admits solutions of finite Morse index in these ranges. In the borderline case $N=p$, we prove that the solutions of $-\Delta_N^H u=e^u$ that are stable outside a compact set are precisely the finite-mass solutions classified by Ciraolo and Li. A main ingredient is independent of stability: if $N>q\tau$, where $\tau=\frac{p}{q+p-1}$, a positive exterior solution of $\Delta_p^H u=u^{-q}$ cannot remain above $\theta$ times the singular solution near infinity for any $\theta>1$. An analogous obstruction holds for the exponential equation. Both obstructions are proved by comparison with explicit radial barriers having frozen right-hand sides. We also prove sharpness when $N>p$, and for the negative-power equation when $N=p>2$. Finally, for operators $\operatorname{div}(B'(H(\nabla u))\nabla H(\nabla u))$ with $p$-growth, we establish the integral estimates needed for the stability arguments with fully justified truncation and absorption procedures. These arguments also repair gaps in earlier proofs for the corresponding isotropic and anisotropic problems.
\end{abstract}

\maketitle

\section{Introduction}

We study positive solutions of
\begin{equation}\label{eq:sing}
\Delta_B^H u=u^{-q},\qquad u>0\qquad\text{in }\mathbb{R}^N,
\end{equation}
with $q>0$, and solutions of the anisotropic Liouville equation
\begin{equation}\label{eq:exp}
-\Delta_B^H u=e^u\qquad\text{in }\mathbb{R}^N,
\end{equation}
where $N\ge2$ and
\[
\Delta_B^H u=\operatorname{div}\big(B'(H(\nabla u))\nabla H(\nabla u)\big).
\]
We assume that $H$ is a uniformly convex Finsler norm, that is,
\begin{enumerate}[label=(H)]
\item $H$ is a norm on $\mathbb{R}^N$ of class $C^2(\mathbb{R}^N\setminus\{0\})$, and the set $\{\xi\in\mathbb{R}^N:H(\xi)<1\}$ is uniformly convex,
\end{enumerate}
and that $B$ has the following growth.
\begin{enumerate}[label=(B)]
\item $B\in C^2((0,\infty))\cap C^1([0,\infty))$, $B(0)=B'(0)=0$, and there are $p\ge2$, $\kappa\in[0,1]$ and $0<\gamma\le\Gamma$ such that, for all $t>0$,
\[
\gamma(\kappa+t)^{p-2}t\le B'(t)\le\Gamma(\kappa+t)^{p-2}t,\qquad
\gamma(p-1)(\kappa+t)^{p-2}\le B''(t)\le\Gamma(p-1)(\kappa+t)^{p-2}.
\]
\end{enumerate}
The model case is $B(t)=t^p/p$, for which $\Delta_B^H$ is the Finsler $p$-Laplacian $\Delta_p^Hu=\operatorname{div}(H(\nabla u)^{p-1}\nabla H(\nabla u))$. When, moreover, $H(\xi)=|\xi|$, it is the $p$-Laplacian. Throughout we write
\[
\lambda:=\frac{\Gamma}{\gamma}(p-1),\qquad
\alpha_0:=-1-\frac{2q+2\sqrt{q(q+\lambda)}}{\lambda},\qquad
\tau:=\frac{p}{q+p-1}.
\]

A function $u\in W^{1,p}_{\rm loc}(\mathbb{R}^N)$ is a weak solution of \eqref{eq:sing} if $u>0$ a.e., $u^{-q}\in L^1_{\rm loc}(\mathbb{R}^N)$ and
\[
\int_{\mathbb{R}^N}B'(H(\nabla u))\langle\nabla H(\nabla u),\nabla\varphi\rangle=-\int_{\mathbb{R}^N}u^{-q}\varphi\qquad\text{for all }\varphi\in C^1_c(\mathbb{R}^N),
\]
and a weak solution of \eqref{eq:exp} if $e^u\in L^1_{\rm loc}(\mathbb{R}^N)$ and the same identity holds with $-u^{-q}$ replaced by $e^u$. Here and below the vector field $B'(H(\xi))\nabla H(\xi)$ is set equal to $0$ at $\xi=0$. Both equations have the form $-\Delta_B^Hu=g(u)$ with $g$ increasing, namely $g(t)=-t^{-q}$ and $g(t)=e^t$, and their linearized operators lead to the quadratic form
\[
Q_u(\varphi):=\int_{\mathbb{R}^N}B''(H(\nabla u))\langle\nabla H(\nabla u),\nabla\varphi\rangle^2+\int_{\mathbb{R}^N}B'(H(\nabla u))\langle D^2H(\nabla u)\nabla\varphi,\nabla\varphi\rangle .
\]
The matrix $\mathcal M(\xi):=B''(H(\xi))\nabla H(\xi)\otimes\nabla H(\xi)+B'(H(\xi))D^2H(\xi)$ is the Hessian of $B\circ H$ at $\xi\ne0$, and $Q_u(\varphi)=\int\langle\mathcal M(\nabla u)\nabla\varphi,\nabla\varphi\rangle$. At $\xi=0$, if $\kappa>0$ we fix any symmetric positive-semidefinite value satisfying $|\mathcal M(0)|\le C\kappa^{p-2}$. If $p=2$ and $\kappa=0$, we fix any bounded symmetric positive-semidefinite value. If $\kappa=0$ and $p>2$, we set $\mathcal M(0)=0$, which is the limit of $\mathcal M(\xi)$ as $\xi\to0$. Stability below is understood with this convention. The proofs work for every such choice; when the critical set has measure zero, the choice is immaterial.

\begin{definition}
A weak solution $u$ of \eqref{eq:sing}, respectively \eqref{eq:exp}, is \emph{stable outside a closed set} $K\subset\mathbb{R}^N$ if
\[
Q_u(\varphi)\ge q\int_{\mathbb{R}^N}u^{-q-1}\varphi^2,\qquad\text{respectively}\qquad Q_u(\varphi)\ge\int_{\mathbb{R}^N}e^u\varphi^2,
\]
for every $\varphi\in C^1_c(\mathbb{R}^N\setminus K)$. It is \emph{stable} if this holds with $K=\emptyset$. Its \emph{Morse index} is the supremum of the dimensions of subspaces of $C^1_c(\mathbb{R}^N)$ on which the corresponding quadratic form
\[
\mathcal Q_u^{\rm sing}(\varphi):=Q_u(\varphi)-q\int u^{-q-1}\varphi^2,
\qquad
\mathcal Q_u^{\rm exp}(\varphi):=Q_u(\varphi)-\int e^u\varphi^2
\]
is negative definite. Thus a stable solution has Morse index zero. The index may also be infinite.
\end{definition}

A solution of finite Morse index is stable outside a compact set. Indeed, if $u$ has Morse index $k<\infty$ and is not stable outside any ball, we can choose $\varphi_1,\dots,\varphi_{k+1}$ with pairwise disjoint compact supports, each giving a negative value to the relevant quadratic form. Because the form is local, these functions span a $(k+1)$-dimensional subspace on which it is negative definite, a contradiction.

For the Laplacian, Farina \cite{Far07} proved that $-\Delta u=e^u$ has no stable solution in $\mathbb{R}^N$ when $2\le N\le9$, and classified the solutions in $\mathbb{R}^2$ that are stable outside a compact set. Dancer and Farina \cite{DF09} then proved that there is no solution stable outside a compact set when $3\le N\le9$. The threshold $N=10$ is sharp: for $N\ge10$ the singular solution $\log(2(N-2))-2\log|x|$ is stable. The second author proved that $-\Delta_pu=e^u$ has no stable solution of class $C^1$ when $N<\frac{p(p+3)}{p-1}$ \cite{Le16}, and studied more general convex nonlinearities and weights in \cite{Le16b,LLL18}. For the anisotropic Laplacian, Fazly and Li \cite{FL21} proved that $-\Delta_2^Hu=e^u$ has no solution stable outside a compact set when $3\le N\le9$, under an additional structural assumption on $H$, and that for $N=2$ the explicit finite-mass solutions are stable outside a compact set. For $N=p$ the solutions of $-\Delta_N^Hu=e^u$ with $e^u\in L^1(\mathbb{R}^N)$ were classified by Esposito \cite{Esp18} in the isotropic case and by Ciraolo and Li \cite{CL24} in general.

For the equation $\Delta u=u^{-q}$, Du and Guo \cite{DG09} found a critical power $q_c(N,2)$ at or below which radial solutions are stable, and the nonexistence of solutions of finite Morse index for $q>q_c(N,2)$ was obtained by Du and Guo and by D\'avila and Ye \cite{DY13}; see \cite[Section~1]{GM16}. Guo and Mei \cite{GM16} extended the radial analysis to the $p$-Laplacian. For $p<N<\frac{p(p+3)}{p-1}$, the exponent $q_c(N,p)$ is the unique $q>0$ such that
\begin{equation}\label{eq:qc}
\big[N-(q-1)\tau\big]^2=\frac{4p}{p-1}\,(N-q\tau),
\end{equation}
and we set $q_c(N,p):=+\infty$ for $N\ge\frac{p(p+3)}{p-1}$. For $N=p$ we have $N-q\tau=\tau(p-1)$ and $N-(q-1)\tau=p\tau$, so \eqref{eq:qc} reads $p\tau=4$, that is, $q=\frac{(p-2)^2}4$. We put $q_c(p,p):=\frac{(p-2)^2}4$. This differs from the convention $q_c(p,p)=0$ of \cite{GM16,HL26a}, and Theorem~\ref{thm:sharp}(a) below shows that the value $\frac{(p-2)^2}4$ is the correct one when $p>2$. Guo and Mei showed, for $p<N<\frac{p(p+3)}{p-1}$, that positive radial solutions of $\Delta_pu=u^{-q}$ are stable when $q\le q_c(N,p)$ and have infinite Morse index when $q>q_c(N,p)$ \cite[Theorem~1.1]{GM16}. For $p>2$ one step of the stability proof needs a correction, which we give in Lemma~\ref{lem:z}. They also stated the corresponding Liouville result for general positive solutions of class $C^{1,\omega}_{\rm loc}$ with finite Morse index when $2<p\le N$ and $q>q_c(N,p)$, with their convention $q_c(p,p)=0$ \cite[Theorem~1.2]{GM16}. In the anisotropic setting, Farina, Sciunzi, and Vuono \cite{FSV24} proved Liouville theorems, including results for solutions stable outside a compact set, for equations of Lane--Emden type. See also the recent work of Liang, Wu, and Yan \cite{LWY25} on anisotropic $p$-Laplace equations with general semilinear terms. In \cite{HL26a,HL26b} we considered \eqref{eq:sing} and \eqref{eq:exp} with general $B$.

Solutions that are only stable outside a compact set are much harder to handle. Stability far out gives integral bounds, and a Harnack inequality turns them into a pointwise estimate on the growth or decay of $u$. One then has to show that no solution can behave in this way. For $p=2$ this can be done with spherical averages, which satisfy a useful differential inequality \cite{DF09}. No analogous scalar differential inequality is available for the $p$-Laplacian when $p\ne2$. Guo and Mei \cite{GM16}, and we in \cite{HL26a,HL26b}, compared $u$ with radial solutions $U_n$ of the same equation on annuli $\{R_*<|x|<n\}$, obtained by monotone iteration, and let $n\to\infty$. The difficulty is that the nonlinearities $-t^{-q}$ and $e^t$ are increasing, so the comparison principle does not compare two solutions, and nothing controls the radial solutions as $n\to\infty$. In \cite[proof of Proposition~3.2]{GM16} the needed control of $U_n'(R_*)$ is derived from the passage from (3.22) to (3.24). This passage requires the lower bound $\int_0^1|\nabla u+t\nabla(U_n-u)|^{p-2}dt\ge|\nabla u|^{p-2}$ on the set where $U_n>u$. The estimate (3.23) there controls $\nabla U_n$ but not $\nabla u$, and we have not been able to justify this step. In \cite{HL26a} and \cite{HL26b} a maximality, respectively minimality, argument was used to make the radial solutions monotone in $n$. That monotonicity argument does not apply as stated, and the finiteness of the limiting radial profile was not established. As far as we know, no complete proof of this nonexistence statement is available for $p>2$, even for the $p$-Laplacian.

We avoid the auxiliary nonlinear problems altogether. A positive solution of $\Delta_p^Hu=u^{-q}$ in an exterior domain is a supersolution of the problem in which the nonlinearity is frozen at the minimum of $u$ on each Wulff sphere $\{H_0=r\}$, where $H_0$ is the dual norm of $H$. That problem has a fixed right-hand side, its radial solutions are explicit, and the comparison principle applies to it without any monotonicity. Comparing $u$ with these barriers on larger and larger annuli, we obtain the following statement about exterior solutions. It does not involve stability. When $N>q\tau$, let
\[
c_s:=\big[\tau^{p-1}(N-q\tau)\big]^{-\frac1{q+p-1}} .
\]
The functions
\[
U_q(x):=c_sH_0(x)^{\tau}\quad(N>q\tau)\qquad\text{and}\qquad U_e(x):=\log\big(p^{p-1}(N-p)\big)-p\log H_0(x)\quad(N>p)
\]
are singular solutions of $\Delta_p^Hu=u^{-q}$ and of $-\Delta_p^Hu=e^u$, respectively.

\begin{theorem}\label{thm:exterior}
Assume that $(H)$ holds and $p\ge2$. Let $R_0\ge0$ and $E:=\{x\in\mathbb{R}^N:H_0(x)>R_0\}$.
\begin{enumerate}[label=(\alph*)]
\item Let $q>0$ and $N>q\tau$. If $u\in W^{1,p}_{\rm loc}(E)\cap C(E)$ is a positive weak solution of $\Delta_p^Hu=u^{-q}$ in $E$, then
\[
\liminf_{r\to\infty}\ r^{-\tau}\min_{\{H_0=r\}}u\ \le\ c_s .
\]
\item Let $N>p$. If $u\in W^{1,p}_{\rm loc}(E)\cap C(E)$ is a weak solution of $-\Delta_p^Hu=e^u$ in $E$, then
\[
\limsup_{r\to\infty}\ r^{p}\max_{\{H_0=r\}}e^u\ \ge\ p^{p-1}(N-p) .
\]
\end{enumerate}
Both inequalities become equalities for $U_q$ and $U_e$.
\end{theorem}

In other words, an exterior solution of the first equation cannot satisfy $u\ge\theta U_q$ near infinity for some $\theta>1$, and an exterior solution of the second cannot satisfy $e^u\le\theta e^{U_e}$ near infinity for some $\theta<1$. The constants are those of the singular solutions, so Theorem~\ref{thm:exterior} cannot be improved. The condition $N>q\tau$ holds whenever $N\ge p$. For the Laplacian, part (b) says that a solution of $-\Delta u=e^u$ outside a ball of $\mathbb{R}^N$, $N\ge3$, satisfies $\limsup_{|x|\to\infty}|x|^2e^{u(x)}\ge2(N-2)$. For the second equation the barrier argument is applied to $v=-u$. The proof uses no upper bound on $u$, and it uses $p\ge2$ only through the subadditivity of $s\mapsto s^{1/(p-1)}$. A different use of explicit barriers, aimed at improving the decay of solutions of the Lane--Emden equation, appears in \cite{Le26}.

Stability enters only through a decay lemma: a solution stable outside a compact set lies above $U_q$ by an arbitrarily large factor, respectively $e^u$ lies below $e^{U_e}$ by an arbitrarily large factor. Together with Theorem~\ref{thm:exterior} this gives our main nonexistence result.

\begin{theorem}\label{thm:outside}
Assume that $(H)$ holds and $p\ge2$.
\begin{enumerate}[label=(\alph*)]
\item If $p\le N<\frac{p(p+3)}{p-1}$ and $q>q_c(N,p)$, then $\Delta_p^Hu=u^{-q}$ has no positive weak solution in $\mathbb{R}^N$ that is stable outside a compact set.
\item If $p<N<\frac{p(p+3)}{p-1}$, then $-\Delta_p^Hu=e^u$ has no weak solution in $\mathbb{R}^N$ that is stable outside a compact set.
\end{enumerate}
In particular, in these ranges neither equation has a solution of finite Morse index.
\end{theorem}

For $p=2$ and $H(\xi)=|\xi|$, the corresponding negative-power problem was treated by Du and Guo \cite{DG09} and D\'avila and Ye \cite{DY13}, while part (b) is due to Dancer and Farina \cite{DF09}. For anisotropic $H$ and $p=2$, part (b) was proved in \cite{FL21} under an additional structural assumption on $H$, which is not needed here. For $p>2$ we are not aware of earlier complete proofs of either part, even for the $p$-Laplacian.

The decay lemma rests on integral estimates obtained by testing the stability inequality with a power of the solution, an idea that for $p\neq2$ goes back to Damascelli, Farina, Sciunzi, and Valdinoci \cite{DFSV09}. Since powers of $u$ are not admissible test functions when $u$ is not known to be locally bounded, one truncates them, and every absorption of one integral into another must be carried out while the truncation is still in place, where all the integrals are finite. In \cite[Proposition~2.1]{LNN18}, \cite[Proposition~6]{HL26a}, \cite[Proposition~5]{HL26b}, and \cite[Proposition~2.1]{LLL18}, the absorption was performed only after the truncation had been removed, leaving the finiteness of some absorbed integrals unjustified. For locally bounded solutions, as in \cite{Le16}, the difficulty does not arise. We prove the estimates for general $B$, in a form that covers both equations at once (Lemma~\ref{lem:key}). They also give complete proofs of the nonexistence of stable solutions claimed in \cite{HL26a,HL26b}, which we state as Corollary~\ref{cor:stable}.

For $N=p$, Theorem~\ref{thm:outside}(b) fails. Ciraolo and Li \cite{CL24} proved that the weak solutions $u\in W^{1,N}_{\rm loc}(\mathbb{R}^N)$ of $-\Delta_N^Hu=e^u$ with $e^u\in L^1(\mathbb{R}^N)$ are exactly the functions
\begin{equation}\label{eq:bubbles}
U_{\ell,x_0}(x)=\log\frac{N\big(\frac{N^2}{N-1}\big)^{N-1}\ell^N}{\Big[1+\ell^{\frac N{N-1}}H_0(x-x_0)^{\frac N{N-1}}\Big]^N},\qquad\ell>0,\ x_0\in\mathbb{R}^N .
\end{equation}
These functions are stable outside a compact set (see the proof of Theorem~\ref{thm:Nliouville} and \cite[Theorem~3(ii)]{HL26b}). We show that there are no others.

\begin{theorem}\label{thm:Nliouville}
Assume that $(H)$ holds and $N\ge2$. A weak solution of $-\Delta_N^Hu=e^u$ in $\mathbb{R}^N$ is stable outside a compact set if and only if it is one of the functions \eqref{eq:bubbles}. In particular, every solution of finite Morse index is one of these functions.
\end{theorem}

For $N=2$ and $H(\xi)=|\xi|$ this is due to Farina \cite{Far07}. For $H(\xi)=|\xi|$ and every $N\ge2$, the statement was claimed by the second author and coauthors in \cite[Theorem~1.8]{LLL18}. The proof there bounds $\int e^{(1+\alpha)u}$ outside a ball by an integral estimate and lets $\alpha\to0^+$. However, the constant in that estimate is of order $\alpha^{1-N}$, because of the term $b^p|b'|^{1-p}$ in the estimate, so the argument gives no bound uniform in $\alpha$. The proofs for $N=2$ insert a logarithmic cutoff function in the stability inequality. For $N=2$ the quadratic form $Q_u$ does not depend on $u$, and one obtains $\int e^u<\infty$. For $N\ge3$, $Q_u(\varphi)$ is comparable to $\int H(\nabla u)^{N-2}|\nabla\varphi|^2$, and the weight is not controlled a priori. We show that on each dyadic annulus this weight is bounded by a power of the gap $s_j$ between $u+N\log|x|$ and its upper bound coming from the decay lemma, while the mass of the annulus is at most $Ce^{-s_j}$. A weighted logarithmic cutoff function then gives finite mass when the gaps are small often enough, and the masses are summable otherwise.

Finally, the thresholds in Theorem~\ref{thm:outside} are sharp when $N>p$, and for the negative exponent also when $N=p>2$.

\begin{theorem}\label{thm:sharp}
Assume that $(H)$ holds and $p\ge2$.
\begin{enumerate}[label=(\alph*)]
\item Let $N\ge p$ and $q>0$. For every $a>0$ the equation $\Delta_p^Hu=u^{-q}$ has a unique positive solution $u_a\in C^1(\mathbb{R}^N)$ of the form $u_a(x)=v(H_0(x))$ with $u_a(0)=a$. If either $N>p$ or $N=p>2$, and $q\le q_c(N,p)$, in particular if $N\ge\frac{p(p+3)}{p-1}$, then $u_a$ is stable. In these cases $u_a(x)>U_q(x)$ for every $x\ne0$.
\item Let $N>p$. The function $U_e$ belongs to $W^{1,p}_{\rm loc}(\mathbb{R}^N)$ and is a weak solution of $-\Delta_p^Hu=e^u$ in $\mathbb{R}^N$. It is stable if and only if $N\ge\frac{p(p+3)}{p-1}$.
\item Let $N=p>2$ and $q>0$. The function $U_q$ belongs to $W^{1,p}_{\rm loc}(\mathbb{R}^N)$ and is a weak solution of $\Delta_p^Hu=u^{-q}$ in $\mathbb{R}^N$, positive except at the origin. It is stable if and only if $q\le\frac{(p-2)^2}4$.
\end{enumerate}
\end{theorem}

By Theorem~\ref{thm:Nliouville}, the restriction $N>p$ in Theorem~\ref{thm:outside}(b) cannot be removed, and by Theorem~\ref{thm:sharp}(b) neither can the condition $N<\frac{p(p+3)}{p-1}$, in Theorem~\ref{thm:outside}(b) or in the last statement of Corollary~\ref{cor:stable}(c). Theorem~\ref{thm:sharp}(a) shows that $q>q_c(N,p)$ is sharp in Theorem~\ref{thm:outside}(a) and in Corollary~\ref{cor:stable}(b) when $N>p$, and also when $N=p>2$, where both the regular radial solutions and the singular solution are stable for $q\le\frac{(p-2)^2}4$. When $2\le N<p$, the integral method needs $q>\frac{(p-2)^2}4$, which is exactly the condition $\alpha_0<1-p$ for the $p$-Laplacian, and we do not know whether this is sharp.

Our results correct several earlier statements. Theorem~\ref{thm:outside} replaces \cite[Theorem~4]{HL26a} and \cite[Theorem~3(i)]{HL26b}, and gives a complete proof of \cite[Theorem~1.2]{GM16} for $N>p$. Theorem~\ref{thm:Nliouville} gives a proof of \cite[Theorem~1.8]{LLL18}, and Theorem~\ref{thm:sharp}(a) proves the stability claim of \cite[Proposition~3]{HL26a}. Theorem~\ref{thm:sharp}(a) also shows that the case $N=p$ of \cite[Theorems~1.1, 1.2, and 2.2]{GM16} and of \cite[Corollary~2 and Theorem~4]{HL26a} must be restricted to $q>\frac{(p-2)^2}4$ when $p>2$: for $0<q\le\frac{(p-2)^2}4$ the radial solutions $u_a$ are positive, stable, of class $C^{1,1/(p-1)}_{\rm loc}$ when $H(\xi)=|\xi|$ (since $v'(r)\sim(r/(Na^q))^{1/(p-1)}$ as $r\to0$), and lie above $U_q$. In the proof of Theorem~\ref{thm:sharp}(a) we also correct one step in the proof of \cite[Theorem~2.10]{GM16} (see Lemma~\ref{lem:z}). Corollary~\ref{cor:stable} gives complete proofs of \cite[Theorem~1 and Corollary~2]{HL26a} and \cite[Theorem~1 and Corollary~2]{HL26b}, with the changes described in Remark~\ref{rem:HL1}. With constant weights, \cite[Theorem~1.1]{LNN18} uses the convention $q_c(p,p)=0$ and thus asserts that there is no stable solution for $N=p$ and every $q>0$. By Theorem~\ref{thm:sharp}(a), this is false when $p>2$ and $0<q\le\frac{(p-2)^2}4$. The proof of \cite[Proposition~2.1]{LNN18} has the truncation gap described above, and a corrected constant-weight nonexistence statement is given by Corollary~\ref{cor:stable}(b). We do not treat the weighted case, which would require carrying the weights through the absorption at the truncated level. For constant weights, \cite[Proposition~2.1]{LLL18} is replaced by Proposition~\ref{prop:est_exp}, so \cite[Theorem~1.4]{LLL18} with constant weights is contained in Corollary~\ref{cor:stable}(c).

Section~\ref{sec_prelim} collects the facts we need about Finsler norms, radial functions, and local regularity. Section~\ref{sec_integral} contains the integral estimates and their consequences for stable solutions. Theorem~\ref{thm:exterior} is proved in Section~\ref{sec_exterior}, Theorem~\ref{thm:outside} in Section~\ref{sec_outside}, Theorem~\ref{thm:Nliouville} in Section~\ref{sec_N}, and Theorem~\ref{thm:sharp} in Section~\ref{sec_sharp}. We write $C$ for positive constants that may change from line to line, and we often omit $dx$.

\section{Preliminaries}\label{sec_prelim}

Since $H$ is a norm of class $C^2$ away from the origin, there are $0<k_1\le k_2$ with $k_1|\xi|\le H(\xi)\le k_2|\xi|$, and for $\xi\ne0$ and $t\ne0$
\begin{equation}\label{eq:homog}
\langle\nabla H(\xi),\xi\rangle=H(\xi),\qquad \nabla H(t\xi)=\operatorname{sign}(t)\nabla H(\xi),\qquad D^2H(t\xi)=\frac1{|t|}D^2H(\xi),\qquad D^2H(\xi)\xi=0 .
\end{equation}
Consequently there are $M,M_1>0$ such that $|\nabla H(\xi)|\le M$ and $|D^2H(\xi)|\le M_1/H(\xi)$ for $\xi\ne0$. The uniform convexity in $(H)$ is equivalent to the positive definiteness of $D^2H(\xi)$ on the hyperplane $\nabla H(\xi)^\perp$, uniformly on $\{H=1\}$, and also to $D^2(H^2/2)(\xi)\ge\lambda_0I$ for $\xi\ne0$ and some $\lambda_0>0$ \cite{CFV14}. It is not the positive definiteness of $D^2H(\xi)$ itself, which is impossible since $D^2H(\xi)\xi=0$. In particular $H^p$ is strictly convex, and the field $\mathcal A(\xi):=H(\xi)^{p-1}\nabla H(\xi)=\nabla(H^p/p)(\xi)$, which defines $\Delta_p^H$, is strictly monotone:
\begin{equation}\label{eq:monotone}
\langle\mathcal A(\xi)-\mathcal A(\eta),\xi-\eta\rangle>0\qquad\text{for }\xi\ne\eta .
\end{equation}
Since $H$ is even, $\mathcal A(-\xi)=-\mathcal A(\xi)$.

By $(B)$, $B''(t)t\le\lambda B'(t)$ for $t>0$, and for $\xi\neq0$
\begin{equation}\label{eq:Mbound}
|\mathcal M(\xi)|\le C(\kappa+H(\xi))^{p-2},\qquad |B'(H(\xi))\nabla H(\xi)|\le C\big(|\xi|+|\xi|^{p-1}\big).
\end{equation}
Fix $\delta>0$. Writing $t=H(\xi)$, assumption $(B)$ gives, for $t<\delta$,
\begin{equation}\label{eq:small}
c_1t\le B'(t)\le C_1t,\qquad c_3\le B''(t)\le C_3,
\end{equation}
with $c_1=\gamma\kappa^{p-2}$, $C_1=\Gamma(\kappa+\delta)^{p-2}$, $c_3=\gamma(p-1)\kappa^{p-2}$, $C_3=\Gamma(p-1)(\kappa+\delta)^{p-2}$, and, for $t\ge\delta$,
\begin{equation}\label{eq:large}
c_2t^{p-1}\le B'(t)\le C_2t^{p-1},\qquad c_4t^{p-2}\le B''(t)\le C_4t^{p-2},
\end{equation}
with $c_2=\gamma$, $C_2=\Gamma(1+\kappa/\delta)^{p-2}$, $c_4=\gamma(p-1)$, $C_4=\Gamma(p-1)(1+\kappa/\delta)^{p-2}$. Here $p\ge2$ is used. When $\kappa=0$, \eqref{eq:large} holds for every $t>0$.

The dual norm of $H$ is $H_0(x):=\sup\{\langle x,\xi\rangle:H(\xi)=1\}$. In \cite{HL26a,HL26b} we used $\widehat H_0(x)=H_0(-x)$. Since $H$ is even, $\widehat H_0=H_0$.

\begin{lemma}\label{lem:dual}
Assume $(H)$. Then $H_0$ is a norm of class $C^1(\mathbb{R}^N\setminus\{0\})$, and for $x\ne0$
\[
\langle\nabla H_0(x),x\rangle=H_0(x),\qquad H(\nabla H_0(x))=1,\qquad \nabla H(\nabla H_0(x))=\frac{x}{H_0(x)} .
\]
\end{lemma}

These facts are classical; see \cite{CS09,CFV14,CFR20}. We use them in the radial reduction below and in Sections~\ref{sec_N} and~\ref{sec_sharp}. The set $\{H_0=r\}$ is called a Wulff sphere, and for $f\ge0$ we have $\int_{\mathbb{R}^N}f(H_0(x))\,dx=N|\{H_0<1\}|\int_0^\infty f(r)r^{N-1}dr$.

\begin{lemma}\label{lem:radial}
Assume $(H)$, let $0\le a<b\le\infty$, and let $V\in C^1((a,b))$ be such that $r\mapsto r^{N-1}\varphi_p(V'(r))$ belongs to $C^1((a,b))$, where $\varphi_p(s)=|s|^{p-2}s$. Then $W(x):=V(H_0(x))$ satisfies $\mathcal A(\nabla W(x))=\varphi_p(V'(r))\,\frac{x}{r}$ and
\[
\Delta_p^HW(x)=r^{1-N}\big(r^{N-1}\varphi_p(V'(r))\big)'\Big|_{r=H_0(x)}
\]
pointwise in $\{a<H_0<b\}$, and also in the weak sense there.
\end{lemma}

\begin{proof}
We have $\nabla W=V'(r)\nabla H_0(x)$ with $r=H_0(x)$. By \eqref{eq:homog} and Lemma~\ref{lem:dual}, $H(\nabla W)=|V'(r)|$ and, where $V'(r)\neq0$, $\nabla H(\nabla W)=\operatorname{sign}(V'(r))\,x/r$. Hence $\mathcal A(\nabla W)=G(r)\,x/r$ with $G=\varphi_p(V')$, and this holds also where $V'=0$. By the Euler identity $\langle\nabla H_0(x),x\rangle=H_0(x)$, for every $G\in C^1((a,b))$,
\begin{equation}\label{eq:div}
\operatorname{div}\Big(G(H_0(x))\frac{x}{H_0(x)}\Big)=G'(r)+\frac{N-1}{r}G(r)=r^{1-N}\big(r^{N-1}G(r)\big)' .
\end{equation}
The field $G(H_0(x))x/H_0(x)$ is of class $C^1$ in $\{a<H_0<b\}$, so the identity also holds in the weak sense.
\end{proof}

We use the comparison principle with a fixed right-hand side.

\begin{lemma}\label{lem:comparison}
Assume $(H)$, let $\Omega$ be a bounded domain, and let $v_1,v_2\in W^{1,p}(\Omega)\cap C(\overline\Omega)$. Suppose that $f_1,f_2\in L^\infty(\Omega)$ satisfy $f_1\le f_2$, that $-\Delta_p^Hv_1\le f_1$ and $-\Delta_p^Hv_2\ge f_2$ weakly in $\Omega$, and that $v_1\le v_2$ on $\partial\Omega$. Then $v_1\le v_2$ in $\Omega$.
\end{lemma}

\begin{proof}
For $\varepsilon>0$ the function $w_\varepsilon=(v_1-v_2-\varepsilon)^+$ belongs to $W^{1,p}(\Omega)$ and vanishes near $\partial\Omega$, so it is an admissible nonnegative test function for both inequalities. Subtracting them gives
\[
\int_{\{v_1-v_2>\varepsilon\}}\langle\mathcal A(\nabla v_1)-\mathcal A(\nabla v_2),\nabla v_1-\nabla v_2\rangle\le\int_\Omega(f_1-f_2)w_\varepsilon\le0 .
\]
By \eqref{eq:monotone}, $\nabla w_\varepsilon=0$ a.e., so $w_\varepsilon=0$. Letting $\varepsilon\to0$ gives $v_1\le v_2$.
\end{proof}

Finally we recall the local regularity theory of Serrin \cite{Ser64} and Trudinger \cite{Tru67}. It applies to $\Delta_p^H$ because $\langle\mathcal A(\xi),\xi\rangle=H(\xi)^p\ge k_1^p|\xi|^p$ and $|\mathcal A(\xi)|\le Mk_2^{p-1}|\xi|^{p-1}$. By scaling, the constants below depend only on $N$, $p$, $s$, $k_1$, $k_2$, $M$, and an upper bound for the indicated scale-invariant quantity.

\begin{lemma}\label{lem:serrin}
Assume $(H)$, $1<p\le N$, and $s>\frac Np$.
\begin{enumerate}[label=(\roman*)]
\item If $u\in W^{1,p}_{\rm loc}(\Omega)$ is a weak solution of $-\Delta_p^Hu=d(x)\,\varphi_p(u)+f(x)$ in an open set $\Omega$, with $d,f\in L^s_{\rm loc}(\Omega)$, then $u$ is locally bounded and H\"older continuous in $\Omega$.
\item If $w\ge0$ is a weak solution of $\Delta_p^Hw=d(x)w^{p-1}$ in $B_{2R}(y)$ with $d\in L^s(B_{2R}(y))$, then $\sup_{B_R(y)}w\le K_*\inf_{B_R(y)}w$, where $K_*$ depends on $R^{p-\frac Ns}\|d\|_{L^s(B_{2R}(y))}$.
\item If $w\in W^{1,p}(B_{2R}(y))$, $w\ge0$, satisfies $-\Delta_p^Hw\le a(x)w^{p-1}$ weakly in $B_{2R}(y)$ with $0\le a\in L^s(B_{2R}(y))$, then
\[
\sup_{B_R(y)}w\le K_* R^{-\frac Np}\|w\|_{L^p(B_{2R}(y))},
\]
where $K_*$ depends on $R^{p-\frac Ns}\|a\|_{L^s(B_{2R}(y))}$.
\item If $z\ge0$ is a weak solution of $-\Delta_p^Hz=f$ in $B_{2\rho}(y)$ with $|f|\le F$, then $\sup_{B_\rho(y)}z\le K_*\big(\inf_{B_\rho(y)}z+1\big)$, where $K_*$ depends on $\rho$ and $F$.
\end{enumerate}
\end{lemma}

In (i)--(iv) the weak formulations are used with bounded, compactly supported test functions in $W^{1,p}$, which is admissible by density. Parts (ii) and (iv) are \cite[Theorems~5 and~6]{Ser64}. These theorems are stated for a solution in a ball of radius $3R$ with the conclusion in the concentric ball of radius $R$, and the radii $2R$ and $R$ are obtained by a covering argument. Local boundedness in (i) and (iii) is proved in \cite{Ser64} for solutions and in \cite{Tru67} for subsolutions, and H\"older continuity in (i) follows from the Harnack inequality by the classical oscillation argument \cite{Ser64}.

\section{Integral estimates}\label{sec_integral}

Both equations are of the form $-\Delta_B^Hu=g(u)$ with $g'>0$. We prove one estimate for this class and then specialize it. Let $\mathcal O\subset\mathbb{R}$ be an open interval and $g\in C^1(\mathcal O)$ with $g'>0$. A weak solution of $-\Delta_B^Hu=g(u)$ is a function $u\in W^{1,p}_{\rm loc}(\mathbb{R}^N)$ with values in $\mathcal O$ a.e., $g(u)\in L^1_{\rm loc}(\mathbb{R}^N)$, and
\begin{equation}\label{eq:weak}
\int_{\mathbb{R}^N}B'(H(\nabla u))\langle\nabla H(\nabla u),\nabla\varphi\rangle=\int_{\mathbb{R}^N}g(u)\varphi\qquad\text{for all }\varphi\in C^1_c(\mathbb{R}^N).
\end{equation}
It is stable outside a closed set $K$ if $Q_u(\varphi)\ge\int g'(u)\varphi^2$ for all $\varphi\in C^1_c(\mathbb{R}^N\setminus K)$. For a fixed $\delta>0$ we set $\Omega_1:=\{x:H(\nabla u(x))<\delta\}$ if $\kappa>0$, $\Omega_1:=\emptyset$ if $\kappa=0$, and $\Omega_2:=\mathbb{R}^N\setminus\Omega_1$. These are sets of points $x$ defined through $\nabla u$, and the bounds \eqref{eq:small} and \eqref{eq:large} hold at $t=H(\nabla u(x))$ on $\Omega_1$ and $\Omega_2$, respectively.

\begin{lemma}\label{lem:key}
Assume that $(H)$ and $(B)$ hold with $p\ge2$. Let $u$ be a weak solution of $-\Delta_B^Hu=g(u)$ that is stable outside a closed set $K$. Let $a,b\in C^1(\mathcal O)$ satisfy the following conditions:
\begin{enumerate}[label=(\roman*)]
\item $a'$ and $b'$ are bounded on $\mathcal O$, $b\ge0$, and $|a(t)|+b(t)\le C(1+|t|)$ for $t\in\mathcal O$;
\item $b'$ has constant sign and does not vanish on $\mathcal O$;
\item there are $\mu,\nu>0$ with $\lambda\mu<\nu$ such that, for $t\in\mathcal O$,
\[
a'(t)^2=\mu|b'(t)|\qquad\text{and}\qquad |g(t)|\,b(t)\le\nu^{-1}g'(t)a(t)^2 .
\]
\end{enumerate}
Put $\Phi:=a^2+b^2|b'|^{-1}$ and $\Psi:=|a|^p|a'|^{2-p}+b^p|b'|^{1-p}$. Then there is $C>0$, depending only on $p,\gamma,\Gamma,\kappa,\delta,M,M_1,\mu$, and $\nu-\lambda\mu$, such that
\begin{equation}\label{eq:key}
\begin{aligned}
\int_{\mathbb{R}^N}g'(u)a(u)^2\psi^p+\int_{\mathbb{R}^N}B'(H(\nabla u))H(\nabla u)|b'(u)|\psi^p\\
\le C\int_{\Omega_1}\Phi(u)\psi^{p-2}|\nabla\psi|^2+C\int_{\Omega_2}\Psi(u)|\nabla\psi|^p
\end{aligned}
\end{equation}
for every nonnegative $\psi\in C^1_c(\mathbb{R}^N\setminus K)$.
\end{lemma}

Note that $a'\ne0$ by (ii) and (iii), so $\Psi$ is well defined.

\begin{proof}
Fix $\psi$ and denote the right-hand side of \eqref{eq:key} without the constant by $\mathcal E(\psi)$. We may assume $\mathcal E(\psi)<\infty$. Put
\[
\mathcal I(\psi):=\int_{\mathbb{R}^N}g'(u)a(u)^2\psi^p,\qquad \mathcal J(\psi):=\int_{\mathbb{R}^N}B'(H(\nabla u))H(\nabla u)|b'(u)|\psi^p .
\]
Since $b'$ is bounded and $B'(t)t\le C(t^2+t^p)$, $\mathcal J(\psi)<\infty$.

\medskip\noindent\textit{Step 1: admissibility.} Let $\phi:=a(u)\psi^{p/2}$. By (i), $a(u)\in L^p_{\rm loc}$ and $\nabla a(u)=a'(u)\nabla u\in L^p_{\rm loc}$, and since $p\ge2$ the function $\psi^{p/2}$ is of class $C^1$. Thus $\phi\in W^{1,p}(\mathbb{R}^N)$ has compact support in $\mathbb{R}^N\setminus K$. Mollifying, we find $\phi_j\in C^1_c(\mathbb{R}^N\setminus K)$, supported in a fixed compact set, with $\phi_j\to\phi$ in $W^{1,p}$ and a.e. By \eqref{eq:Mbound}, $|\mathcal M(\nabla u)|\le C(\kappa+|\nabla u|)^{p-2}\in L^{p/(p-2)}_{\rm loc}$, so $Q_u(\phi_j)\to Q_u(\phi)<\infty$, while Fatou's lemma gives $\int g'(u)\phi^2\le\liminf_j\int g'(u)\phi_j^2$. Therefore
\begin{equation}\label{eq:stab_phi}
\mathcal I(\psi)=\int_{\mathbb{R}^N}g'(u)\phi^2\le Q_u(\phi)<\infty .
\end{equation}
Next we show that \eqref{eq:weak} holds with $\varphi=b(u)\psi^p$. Let $T_n(s)=\min\{s,n\}$ and $\varphi_n:=T_n(b(u))\psi^p$. This function is bounded, belongs to $W^{1,p}$, and has compact support, so by mollification and \eqref{eq:Mbound} it is admissible in \eqref{eq:weak}. Since $\nabla\varphi_n=\chi_{\{b(u)<n\}}b'(u)\psi^p\nabla u+pT_n(b(u))\psi^{p-1}\nabla\psi$ and $\langle\nabla H(\nabla u),\nabla u\rangle=H(\nabla u)$, we get
\begin{align*}
\int_{\{b(u)<n\}}B'(H(\nabla u))H(\nabla u)b'(u)\psi^p+p\int B'(H(\nabla u))\langle\nabla H(\nabla u),\nabla\psi\rangle T_n(b(u))\psi^{p-1}\\
=\int g(u)T_n(b(u))\psi^p .
\end{align*}
The three integrands are dominated, respectively, by
\[
B'(H(\nabla u))H(\nabla u)|b'(u)|\psi^p,\qquad CB'(H(\nabla u))b(u)|\nabla\psi|,\qquad |g(u)|b(u)\psi^p,
\]
which are integrable: the first because $\mathcal J(\psi)<\infty$, the second by \eqref{eq:Mbound} and $b(u)\in L^p_{\rm loc}$, and the third because $|g(u)|b(u)\le\nu^{-1}g'(u)a(u)^2$ and \eqref{eq:stab_phi}. Letting $n\to\infty$ and multiplying by the sign of $b'$, we obtain
\begin{equation}\label{eq:test_b}
\mathcal J(\psi)\le pM\int_{\mathbb{R}^N}B'(H(\nabla u))b(u)\psi^{p-1}|\nabla\psi|+\int_{\mathbb{R}^N}|g(u)|b(u)\psi^p .
\end{equation}

\medskip\noindent\textit{Step 2: the stability inequality.} We have $\nabla\phi=a'(u)\psi^{p/2}\nabla u+\frac p2a(u)\psi^{\frac{p-2}2}\nabla\psi$. By \eqref{eq:homog},
\[
\langle\nabla H(\nabla u),\nabla\phi\rangle=a'(u)H(\nabla u)\psi^{p/2}+\tfrac p2a(u)\psi^{\frac{p-2}2}\langle\nabla H(\nabla u),\nabla\psi\rangle,
\]
and, since $D^2H(\nabla u)\nabla u=0$,
\[
\langle D^2H(\nabla u)\nabla\phi,\nabla\phi\rangle=\tfrac{p^2}4a(u)^2\psi^{p-2}\langle D^2H(\nabla u)\nabla\psi,\nabla\psi\rangle .
\]
On $\{\nabla u=0\}$ the integrand of $Q_u(\phi)$ is $\langle\mathcal M(0)\nabla\phi,\nabla\phi\rangle$ with $\nabla\phi=\frac p2a(u)\psi^{\frac{p-2}2}\nabla\psi$. It vanishes if $\kappa=0$ and $p>2$. Otherwise it is at most $Ca(u)^2\psi^{p-2}|\nabla\psi|^2$, and either $\kappa>0$, in which case $\{\nabla u=0\}\subset\Omega_1$ and $a^2\le\Phi$, or $\kappa=0$ and $p=2$, in which case $\Omega_2=\mathbb{R}^N$ and $a^2\le\Psi$. Its contribution is therefore at most $C\mathcal E(\psi)$, and in the integrals below we integrate over $\{\nabla u\ne0\}$. Writing $H=H(\nabla u)$, $B'=B'(H(\nabla u))$, $B''=B''(H(\nabla u))$, and using $|\nabla H|\le M$ and $|D^2H(\nabla u)|\le M_1/H$, we get from \eqref{eq:stab_phi}
\[
\mathcal I(\psi)\le I_1+I_2+I_3+I_4+C\mathcal E(\psi),
\]
where
\begin{gather*}
I_1=\int B''H^2a'(u)^2\psi^p,\qquad I_2=\frac{p^2M^2}4\int B''a(u)^2\psi^{p-2}|\nabla\psi|^2,\\
I_3=pM\int B''H|a'(u)||a(u)|\psi^{p-1}|\nabla\psi|,\qquad I_4=\frac{p^2M_1}4\int\frac{B'}{H}a(u)^2\psi^{p-2}|\nabla\psi|^2 .
\end{gather*}
Let $\varepsilon\in(0,1)$. On $\Omega_1$, \eqref{eq:small} gives $B''\le C_3$, $B'/H\le C_1$, and $c_3H^2\le B''H^2$, and Young's inequality gives
\[
H|a'||a|\psi^{p-1}|\nabla\psi|\le\varepsilon' H^2a'^2\psi^p+C_{\varepsilon'}a^2\psi^{p-2}|\nabla\psi|^2 .
\]
On $\Omega_2$, \eqref{eq:large} gives $B''\le C_4H^{p-2}$, $B'/H\le C_2H^{p-2}$, and $c_4H^p\le B''H^2$. If $p>2$, Young's inequality with exponents $\frac p{p-2}$ and $\frac p2$ gives
\[
H^{p-2}a^2\psi^{p-2}|\nabla\psi|^2=\big(H^pa'^2\psi^p\big)^{\frac{p-2}p}\big(|a|^p|a'|^{2-p}|\nabla\psi|^p\big)^{\frac2p}\le\varepsilon'H^pa'^2\psi^p+C_{\varepsilon'}|a|^p|a'|^{2-p}|\nabla\psi|^p,
\]
and this is an identity with $\varepsilon'=0$, $C_{\varepsilon'}=1$ if $p=2$. In the same way,
\[
H^{p-1}|a'||a|\psi^{p-1}|\nabla\psi|\le\varepsilon'H^pa'^2\psi^p+C_{\varepsilon'}H^{p-2}a^2\psi^{p-2}|\nabla\psi|^2 .
\]
Since $a^2\le\Phi$ and $|a|^p|a'|^{2-p}\le\Psi$, choosing $\varepsilon'$ small in terms of $\varepsilon$ we obtain $I_j\le\varepsilon I_1+C_\varepsilon\mathcal E(\psi)$ for $j=2,3,4$. Finally $B''H^2\le\lambda B'H$ and $a'^2=\mu|b'|$, so $I_1\le\lambda\mu\mathcal J(\psi)$. Hence
\begin{equation}\label{eq:step1}
\mathcal I(\psi)\le(1+3\varepsilon)\lambda\mu\mathcal J(\psi)+C_\varepsilon\mathcal E(\psi).
\end{equation}

\medskip\noindent\textit{Step 3: the equation.} We estimate the first term on the right of \eqref{eq:test_b}. On $\Omega_1$, by \eqref{eq:small} and Young's inequality,
\[
pMB'b\psi^{p-1}|\nabla\psi|\le pMC_1Hb\psi^{p-1}|\nabla\psi|\le\varepsilon c_1H^2|b'|\psi^p+C_\varepsilon b^2|b'|^{-1}\psi^{p-2}|\nabla\psi|^2,
\]
and $c_1H^2|b'|\le B'H|b'|$. On $\Omega_2$, by \eqref{eq:large} and Young's inequality with exponents $\frac p{p-1}$ and $p$,
\[
pMB'b\psi^{p-1}|\nabla\psi|\le pMC_2H^{p-1}b\psi^{p-1}|\nabla\psi|\le\varepsilon c_2H^p|b'|\psi^p+C_\varepsilon b^p|b'|^{1-p}|\nabla\psi|^p,
\]
and $c_2H^p|b'|\le B'H|b'|$. Inserting these bounds in \eqref{eq:test_b}, absorbing $\varepsilon\mathcal J(\psi)<\infty$, and using $|g|b\le\nu^{-1}g'a^2$, we get
\begin{equation}\label{eq:step2}
(1-\varepsilon)\mathcal J(\psi)\le C_\varepsilon\mathcal E(\psi)+\nu^{-1}\mathcal I(\psi).
\end{equation}

\medskip\noindent\textit{Step 4: conclusion.} From \eqref{eq:step1} and \eqref{eq:step2},
\[
\Big(1-\frac{1+3\varepsilon}{1-\varepsilon}\cdot\frac{\lambda\mu}{\nu}\Big)\mathcal I(\psi)\le C_\varepsilon\mathcal E(\psi).
\]
Since $\lambda\mu<\nu$, we can fix $\varepsilon$, depending on $\lambda\mu/\nu$, so that the coefficient is positive. As $\mathcal I(\psi)<\infty$ by \eqref{eq:stab_phi}, we obtain $\mathcal I(\psi)\le C\mathcal E(\psi)$, and then $\mathcal J(\psi)\le C\mathcal E(\psi)$ by \eqref{eq:step2}.
\end{proof}

We now choose the truncations. For \eqref{eq:sing} we have $\mathcal O=(0,\infty)$ and $g(t)=-t^{-q}$, so $g'(t)=qt^{-q-1}$.

\begin{proposition}\label{prop:est_sing}
Assume that $(H)$ and $(B)$ hold with $p\ge2$, let $q>0$, and let $u$ be a weak solution of \eqref{eq:sing} that is stable outside a closed set $K$. Assume $\alpha_0<1-p$ and let $\alpha\in(\alpha_0,1-p]$. Then
\begin{equation}\label{eq:est_sing}
\int_{\mathbb{R}^N}\Big(B'(H(\nabla u))H(\nabla u)u^{\alpha-1}+u^{\alpha-q}\Big)\eta^{pm}
\le C\int_{\Omega_1}|\nabla\eta|^{\frac{2(q-\alpha)}{q+1}}+C\int_{\Omega_2}|\nabla\eta|^{\frac{p(q-\alpha)}{q+p-1}}
\end{equation}
for every $\eta\in C^1_c(\mathbb{R}^N\setminus K)$ with $0\le\eta\le1$ and every
\begin{equation}\label{eq:m_sing}
m\ge\max\Big\{\frac{q-\alpha}{q+p-1},\ \frac{2(q-\alpha)}{p(q+1)}\Big\},
\end{equation}
with $C$ depending on $m$ but not on $\eta$. When $p=2$ and $\alpha=-1$, the first term on the left must be omitted.
\end{proposition}

\begin{proof}
If $p=2$ and $\alpha=-1$, both exponents on the right equal $2$ and $m\ge1$. The stability inequality with $\varphi=\eta^m$ and \eqref{eq:Mbound} give $q\int u^{-q-1}\eta^{2m}\le C\int m^2\eta^{2m-2}|\nabla\eta|^2\le C\int|\nabla\eta|^2$, which is \eqref{eq:est_sing}. So let $\alpha\ne-1$. Then $\alpha<-1$, and we use
\[
\mu:=\frac{(\alpha+1)^2}{4|\alpha|},\qquad \theta_1:=\frac{-1-\alpha}{q-\alpha},\qquad \theta_2:=\frac{1-p-\alpha}{q-\alpha}.
\]
Both $\theta_1$ and $\theta_2$ lie in $[0,1)$.

The choice of $\alpha$ gives $\lambda\mu<q$. Indeed, with $s=|\alpha|$, the inequality $\lambda(\alpha+1)^2<4q|\alpha|$ reads $\lambda s^2-2(\lambda+2q)s+\lambda<0$, that is, $s_-<s<s_+$, where $s_\pm$ are the roots. Since $s_+=-\alpha_0$ and $s_-s_+=1$, we have $s_-<1<s$ and $s<s_+$.

For $k\in\mathbb{N}$, let $a_k$ and $b_k$ be equal to $t^{\frac{\alpha+1}2}$ and $t^\alpha$ for $t\ge\frac1k$, and to their tangent lines at $t=\frac1k$ for $0<t<\frac1k$:
\[
a_k(t)=k^{-\frac{\alpha+1}2}+\tfrac{\alpha+1}2k^{\frac{1-\alpha}2}\big(t-\tfrac1k\big),\qquad
b_k(t)=k^{-\alpha}+\alpha k^{1-\alpha}\big(t-\tfrac1k\big)\qquad(0<t<\tfrac1k).
\]
These functions are positive, decreasing, of class $C^1$, bounded with bounded derivatives, and $a_k'^2=\mu|b_k'|$. A direct computation gives
\[
a_k(t)^2-tb_k(t)=\frac{(\alpha-1)^2}{4k^{\alpha+1}}(1-kt)^2\qquad\text{for }0<t<\tfrac1k,
\]
and $a_k^2=tb_k$ for $t\ge\frac1k$. Hence $t^{-q}b_k\le t^{-q-1}a_k^2=q^{-1}g'a_k^2$, and Lemma~\ref{lem:key} applies with $\nu=q$. Put $F_k(t):=t^{-q-1}a_k(t)^2$ and let $\Phi_k$ and $\Psi_k$ be the functions $\Phi$ and $\Psi$ of Lemma~\ref{lem:key} for $a_k,b_k$. We claim that
\begin{equation}\label{eq:trunc_sing}
\Phi_k\le C F_k^{\theta_1},\qquad \Psi_k\le CF_k^{\theta_2}\qquad\text{on }(0,\infty),
\end{equation}
with $C=C(p,q,\alpha)$. For $t\ge\frac1k$, $F_k(t)=t^{\alpha-q}$, $\Phi_k(t)=(1+|\alpha|^{-1})t^{\alpha+1}$, and $\Psi_k(t)=\big((\frac{|\alpha+1|}2)^{2-p}+|\alpha|^{1-p}\big)t^{\alpha+p-1}$, while $(\alpha-q)\theta_1=\alpha+1$ and $(\alpha-q)\theta_2=\alpha+p-1$. For $0<t<\frac1k$, the bounds $k^{-\frac{\alpha+1}2}\le a_k\le\frac{1-\alpha}2k^{-\frac{\alpha+1}2}$, $b_k\le(1-\alpha)k^{-\alpha}$, $|a_k'|=\frac{|\alpha+1|}2k^{\frac{1-\alpha}2}$, and $|b_k'|=|\alpha|k^{1-\alpha}$ give $F_k\ge k^{q-\alpha}$, $\Phi_k\le Ck^{-1-\alpha}$, and $\Psi_k\le Ck^{1-p-\alpha}$, which is \eqref{eq:trunc_sing}.

Apply Lemma~\ref{lem:key} with $\psi=\eta^m$. Note that $m\ge1$ by \eqref{eq:m_sing}, since $q-\alpha\ge q+p-1$. By \eqref{eq:trunc_sing} and Young's inequality with exponents $\frac1{\theta_i}$ and $\frac1{1-\theta_i}$,
\begin{align*}
Cm^2\Phi_k(u)\eta^{pm-2}|\nabla\eta|^2&\le\tfrac q4F_k(u)\eta^{\frac{pm-2}{\theta_1}}+C|\nabla\eta|^{\frac{2(q-\alpha)}{q+1}},\\
Cm^p\Psi_k(u)\eta^{p(m-1)}|\nabla\eta|^p&\le\tfrac q4F_k(u)\eta^{\frac{p(m-1)}{\theta_2}}+C|\nabla\eta|^{\frac{p(q-\alpha)}{q+p-1}}
\end{align*}
(if $\theta_i=0$ the corresponding bound holds without the first term). Condition \eqref{eq:m_sing} is equivalent to $\frac{pm-2}{\theta_1}\ge pm$ together with $\frac{p(m-1)}{\theta_2}\ge pm$, so both powers of $\eta$ are at most $\eta^{pm}$. The left-hand side of \eqref{eq:key} contains $\int g'(u)a_k(u)^2\eta^{pm}=q\int F_k(u)\eta^{pm}$, which is finite by \eqref{eq:stab_phi}. Absorbing half of it, we obtain
\[
\int_{\mathbb{R}^N}\Big(B'(H(\nabla u))H(\nabla u)|b_k'(u)|+F_k(u)\Big)\eta^{pm}\le C\int_{\Omega_1}|\nabla\eta|^{\frac{2(q-\alpha)}{q+1}}+C\int_{\Omega_2}|\nabla\eta|^{\frac{p(q-\alpha)}{q+p-1}}
\]
with $C$ independent of $k$. For each $t>0$ we have $F_k(t)=t^{\alpha-q}$ and $|b_k'(t)|=|\alpha|t^{\alpha-1}$ as soon as $k\ge1/t$. Fatou's lemma gives \eqref{eq:est_sing}.
\end{proof}

For \eqref{eq:exp} we have $\mathcal O=\mathbb{R}$ and $g(t)=g'(t)=e^t$.

\begin{proposition}\label{prop:est_exp}
Assume that $(H)$ and $(B)$ hold with $p\ge2$, and let $u$ be a weak solution of \eqref{eq:exp} that is stable outside a closed set $K$. Let $\alpha\in(0,\frac4\lambda)$. Then
\begin{equation}\label{eq:est_exp}
\int_{\mathbb{R}^N}\Big(B'(H(\nabla u))H(\nabla u)e^{\alpha u}+e^{(\alpha+1)u}\Big)\eta^{pm}
\le C\int_{\Omega_1}|\nabla\eta|^{2(\alpha+1)}+C\int_{\Omega_2}|\nabla\eta|^{p(\alpha+1)}
\end{equation}
for every $\eta\in C^1_c(\mathbb{R}^N\setminus K)$ with $0\le\eta\le1$ and every $m\ge\alpha+1$, with $C$ depending on $m$ but not on $\eta$.
\end{proposition}

\begin{proof}
For $k\in\mathbb{N}$, let $a_k(t)=e^{\alpha t/2}$ and $b_k(t)=e^{\alpha t}$ for $t<k$, continued by their tangent lines at $t=k$:
\[
a_k(t)=e^{\alpha k/2}(1+\sigma),\qquad b_k(t)=e^{\alpha k}(1+2\sigma),\qquad \sigma:=\tfrac\alpha2(t-k)\qquad(t\ge k).
\]
Then $a_k,b_k\in C^1(\mathbb{R})$ are positive and increasing, have bounded derivatives and at most linear growth, and satisfy $a_k'^2=\frac\alpha4b_k'$. Since $(1+\sigma)^2\ge1+2\sigma$, we have $a_k^2\ge b_k$, so $e^tb_k\le e^ta_k^2$. Thus Lemma~\ref{lem:key} applies with $\mu=\frac\alpha4$ and $\nu=1$, and $\lambda\mu<\nu$ because $\alpha<\frac4\lambda$. Put $F_k(t):=e^ta_k(t)^2$ and $\theta:=\frac{\alpha}{\alpha+1}$. We claim that
\begin{equation}
\Phi_k\le CF_k^\theta,\qquad\Psi_k\le CF_k^\theta\qquad\text{on }\mathbb{R},
\end{equation}
with $C=C(p,\alpha)$. For $t<k$ we have $F_k=e^{(\alpha+1)t}$, $\Phi_k=(1+\frac1\alpha)e^{\alpha t}$, and $\Psi_k=\big((\frac\alpha2)^{2-p}+\alpha^{1-p}\big)e^{\alpha t}$. For $t\ge k$,
\[
F_k^\theta=e^{\alpha k}e^{\frac{2\sigma}{\alpha+1}}(1+\sigma)^{\frac{2\alpha}{\alpha+1}},\qquad
\Phi_k\le\big(1+\tfrac4\alpha\big)e^{\alpha k}(1+\sigma)^2,\qquad
\Psi_k\le\big((\tfrac\alpha2)^{2-p}+2^p\alpha^{1-p}\big)e^{\alpha k}(1+\sigma)^p,
\]
so $\Phi_k/F_k^\theta$ and $\Psi_k/F_k^\theta$ are bounded by $C(1+\sigma)^{p}e^{-\frac{2\sigma}{\alpha+1}}\le C(p,\alpha)$ for $\sigma\ge0$.

We apply Lemma~\ref{lem:key} with $\psi=\eta^m$ and argue as in the proof of Proposition~\ref{prop:est_sing}. Young's inequality with exponents $\frac{\alpha+1}\alpha$ and $\alpha+1$ produces the powers $|\nabla\eta|^{2(\alpha+1)}$ and $|\nabla\eta|^{p(\alpha+1)}$, and the powers of $\eta$ are at most $\eta^{pm}$ provided $pm\ge2(\alpha+1)$ and $m\ge\alpha+1$. Both hold since $m\ge\alpha+1$ and $p\ge2$. After absorption, $F_k(t)=e^{(\alpha+1)t}$ and $b_k'(t)=\alpha e^{\alpha t}$ as soon as $k>t$, and Fatou's lemma gives \eqref{eq:est_exp}.
\end{proof}

Fix $\chi\in C^\infty(\mathbb{R})$ with $0\le\chi\le1$, $\chi=1$ on $(-\infty,1]$, $\chi=0$ on $[2,\infty)$, and $|\chi'|\le2$, and put $\eta_R(x)=\chi(|x|/R)$. Then $|\nabla\eta_R|\le2/R$, and $\nabla\eta_R$ is supported in $B_{2R}\setminus B_R$. For $2\le N<p$ put $q_*(N,p):=\frac{(p-2)^2}4$, and for $N\ge p$ put $q_*(N,p):=q_c(N,p)$.

\begin{corollary}\label{cor:stable}
Assume that $(H)$ and $(B)$ hold with $p\ge2$.
\begin{enumerate}[label=(\alph*)]
\item Let $q>0$ and assume that
\begin{equation}\label{eq:cond_sing}
\alpha_0<1-p\qquad\text{and}\qquad N<\frac{p(q-\alpha_0)}{q+p-1},
\end{equation}
and, if $\kappa>0$, also $N<\frac{2(q-\alpha_0)}{q+1}$. Then \eqref{eq:sing} has no stable weak solution.
\item If $2\le N<\frac{p(p+3)}{p-1}$ and $q>q_*(N,p)$, then $\Delta_p^Hu=u^{-q}$ has no stable weak solution. Indeed, for $\Delta_p^H$ and $2\le N<\frac{p(p+3)}{p-1}$, condition \eqref{eq:cond_sing} is equivalent to $q>q_*(N,p)$.
\item If $N<p(1+\frac4\lambda)$ and, when $\kappa>0$, also $N<2(1+\frac4\lambda)$, then \eqref{eq:exp} has no stable weak solution. For $\Delta_p^H$ this condition reads $N<\frac{p(p+3)}{p-1}$.
\end{enumerate}
\end{corollary}

\begin{proof}
(a) Let $u$ be a stable weak solution. Every $\alpha\in(\alpha_0,1-p]$ close to $\alpha_0$ satisfies $\alpha<-1$, and Proposition~\ref{prop:est_sing} with $\eta=\eta_R$ and $K=\emptyset$ gives
\[
\int_{B_R}u^{\alpha-q}\le C\big(R^{N-\frac{2(q-\alpha)}{q+1}}+R^{N-\frac{p(q-\alpha)}{q+p-1}}\big),
\]
the first term being present only if $\kappa>0$. Both exponents are negative for $\alpha$ close to $\alpha_0$, so $\int_{\mathbb{R}^N}u^{\alpha-q}=0$, which is impossible.

(b) Here $\lambda=p-1$. The function $q\mapsto2q+2\sqrt{q(q+p-1)}$ is increasing and equals $(p-1)(p-2)$ at $q=\frac{(p-2)^2}4$, since then $q+p-1=\frac{p^2}4$. Hence $\alpha_0<1-p$ if and only if $q>\frac{(p-2)^2}4$. If $N\le p$, this implies $p(q-\alpha_0)>p(q+p-1)\ge N(q+p-1)$, and $\frac{(p-2)^2}4=q_*(N,p)$. If $p<N<\frac{p(p+3)}{p-1}$, the second inequality in \eqref{eq:cond_sing} implies the first, and it is equivalent to $N+\tau(\alpha_0-q)<0$. A direct computation shows that this holds if and only if $q>q_c(N,p)$. This is the relation between $q_c$ and $\alpha_0$ (denoted $\gamma_0(q)$ there) used in \cite[Section~3]{GM16}. The nonexistence statement follows from (a).

(c) Proposition~\ref{prop:est_exp} with $\eta=\eta_R$ gives $\int_{B_R}e^{(\alpha+1)u}\le C\big(R^{N-2(\alpha+1)}+R^{N-p(\alpha+1)}\big)$, the first term being present only if $\kappa>0$, and we choose $\alpha<\frac4\lambda$ close to $\frac4\lambda$.
\end{proof}

\begin{remark}\label{rem:HL1}
Part (a) contains \cite[Theorem~1]{HL26a}, except when $\kappa=0$, $N<p$, and $\alpha_0\ge1-p$. Hypothesis (i) there is $\kappa>0$ and $N<\frac{2(q-\alpha_0)}{q+p-1}$. Together with $N\ge2$ it gives $\alpha_0<1-p$, and since $p\ge2$ the number $\frac{2}{q+p-1}$ is at most $\frac{2}{q+1}$ and at most $\frac{p}{q+p-1}$. Hypothesis (ii) there is $\kappa=0$ and the second inequality in \eqref{eq:cond_sing}. It implies $\alpha_0<1-p$ when $N\ge p$ but not when $N<p$, and without $\alpha_0<1-p$ the range of $\alpha$ in the integral estimate is empty. In \cite[Corollary~2]{HL26a} the convention $q_c(N,p)=0$ for $N\le p$ was used. With that convention the statement is false for $N=p>2$ and $0<q\le\frac{(p-2)^2}4$, by Theorem~\ref{thm:sharp}(a), and for $N<p$ its proof requires $q>\frac{(p-2)^2}4$. Part (b) gives the corrected statement. The exponent of $|\nabla\eta|$ on $\Omega_1$ in \eqref{eq:est_sing} differs from the one in \cite[Proposition~6]{HL26a}: the bound $u^{\alpha+1}\le Cu^{\alpha+p-1}$ used there fails for $u<1$ when $p>2$. Part (c) is \cite[Theorem~1 and Corollary~2]{HL26b}.
\end{remark}

\section{Exterior solutions}\label{sec_exterior}

In this section we prove Theorem~\ref{thm:exterior}. Both parts follow from a comparison lemma for subsolutions of $\Delta_p^Hv=h(v)$ with $h$ nonincreasing. For $R_0\ge0$ let $E=\{H_0>R_0\}$.

\begin{lemma}\label{lem:barrier}
Assume that $(H)$ holds and $p\ge2$. Let $\mathcal O\subset\mathbb{R}$ be an open interval and $h:\mathcal O\to(0,\infty)$ continuous and nonincreasing. Let $v\in W^{1,p}_{\rm loc}(E)\cap C(E)$ take values in $\mathcal O$ and satisfy $\Delta_p^Hv\le h(v)$ weakly in $E$, that is,
\[
\int_E\langle\mathcal A(\nabla v),\nabla\varphi\rangle\ge-\int_Eh(v)\varphi\qquad\text{for all }0\le\varphi\in C^1_c(E).
\]
For $r>R_0$ put $m(r):=\min_{\{H_0=r\}}v$, and for $R_1>R_0$ and $r,S>R_1$ put
\[
F(r):=\int_{R_1}^rs^{N-1}h(m(s))\,ds,\qquad \Lambda(S):=\int_{R_1}^S\big(r^{1-N}F(r)\big)^{\frac1{p-1}}dr,\qquad \mathcal L(S):=\int_{R_1}^Sr^{-\frac{N-1}{p-1}}dr .
\]
Then, for every $R_1>R_0$,
\[
\limsup_{S\to\infty}\frac{m(S)-\Lambda(S)}{\mathcal L(S)}<+\infty .
\]
\end{lemma}

\begin{proof}
The function $m$ is continuous on $(R_0,\infty)$, since $v$ is uniformly continuous on compact subsets of $E$ and $m(r)=\min_{\{H_0=1\}}v(r\,\cdot)$. Put $f(r):=h(m(r))$, a continuous positive function. Fix $R_1>R_0$ and suppose, to reach a contradiction, that there are $S_j\to\infty$ with
\begin{equation}\label{eq:ratio}
\frac{m(S_j)-\Lambda(S_j)}{\mathcal L(S_j)}\to+\infty .
\end{equation}

\medskip\noindent\textit{The barriers.} For $S>R_1$ and $t\in\mathbb{R}$ let
\[
V_t(r):=m(R_1)+\int_{R_1}^r\varphi_p^{-1}\big(s^{1-N}(t+F(s))\big)ds,\qquad R_1\le r\le S .
\]
Then $V_t\in C^1([R_1,S])$, and $r^{N-1}\varphi_p(V_t'(r))=t+F(r)$ is of class $C^1$ with derivative $r^{N-1}f(r)$. The value $V_t(S)$ is a continuous and strictly increasing function of $t$ that tends to $\pm\infty$ as $t\to\pm\infty$. Let $t_S$ be the value for which $V_{t_S}(S)=m(S)$, and put $W(x):=V_{t_S}(H_0(x))$ on $A_S:=\{R_1<H_0<S\}$. By Lemma~\ref{lem:radial}, $\Delta_p^HW=f(H_0(x))$ in $A_S$. Since $v\ge m(H_0(x))$ and $h$ is nonincreasing, $h(v(x))\le f(H_0(x))$, hence
\[
-\Delta_p^HW=-f(H_0)\le-h(v)\le-\Delta_p^Hv\qquad\text{weakly in }A_S .
\]
On $\{H_0=R_1\}$ we have $W=m(R_1)\le v$, and on $\{H_0=S\}$ we have $W=m(S)\le v$. Both right-hand sides are bounded on $\overline{A_S}$ because $v$ is continuous there. Lemma~\ref{lem:comparison} gives $W\le v$ in $A_S$.

\medskip\noindent\textit{The slopes.} Since $p\ge2$, $s\mapsto s^{\frac1{p-1}}$ is subadditive on $[0,\infty)$, and $\varphi_p^{-1}$ is increasing. With $t^+=\max\{t,0\}$ and $F\ge0$,
\[
V_t'(r)\le\varphi_p^{-1}\big(r^{1-N}(t^++F(r))\big)\le(t^+)^{\frac1{p-1}}r^{-\frac{N-1}{p-1}}+\big(r^{1-N}F(r)\big)^{\frac1{p-1}} .
\]
Integrating from $R_1$ to $S$ and using $V_{t_S}(S)=m(S)$, we get
\[
m(S)\le m(R_1)+(t_S^+)^{\frac1{p-1}}\mathcal L(S)+\Lambda(S).
\]
Since $\mathcal L(S)\ge\mathcal L(R_1+1)>0$ for $S\ge R_1+1$, \eqref{eq:ratio} gives $t_{S_j}\to+\infty$.

\medskip\noindent\textit{Conclusion.} Fix $x_0$ with $H_0(x_0)=r_0:=R_1+1$. For $S_j>r_0$ with $t_{S_j}>0$, since $F\ge0$,
\[
v(x_0)\ge V_{t_{S_j}}(r_0)\ge m(R_1)+t_{S_j}^{\frac1{p-1}}\int_{R_1}^{r_0}r^{-\frac{N-1}{p-1}}dr\longrightarrow+\infty,
\]
which is absurd.
\end{proof}

\begin{proof}[Proof of Theorem~\ref{thm:exterior}]
(a) Apply Lemma~\ref{lem:barrier} with $v=u$, $\mathcal O=(0,\infty)$, and $h(t)=t^{-q}$; the equation $\Delta_p^Hu=u^{-q}$ gives the required inequality. Suppose that $\liminf_{r\to\infty}r^{-\tau}m(r)>c_s$, and choose $c'>c_s$ and $R_1>R_0$ with $m(r)\ge c'r^\tau$ for $r\ge R_1$. Since $N-q\tau>0$,
\[
F(r)\le c'^{-q}\int_{R_1}^rs^{N-1-q\tau}ds\le\frac{c'^{-q}}{N-q\tau}\,r^{N-q\tau}.
\]
Because $\frac{1-q\tau}{p-1}=\tau-1$, this gives $\Lambda(S)\le c''S^\tau$ with $c'':=\tau^{-1}\big(c'^{q}(N-q\tau)\big)^{-\frac1{p-1}}$. The inequality $c''<c'$ is equivalent to $c'^{q+p-1}>\big(\tau^{p-1}(N-q\tau)\big)^{-1}=c_s^{q+p-1}$, which holds. Hence $m(S)-\Lambda(S)\ge(c'-c'')S^\tau$. On the other hand, $\mathcal L(S)$ is bounded if $N>p$, equals $\log(S/R_1)$ if $N=p$, and is at most $\frac{p-1}{p-N}S^{\frac{p-N}{p-1}}$ if $N<p$. Since $\tau-\frac{p-N}{p-1}=\frac{N-q\tau}{p-1}>0$, in all cases $\mathcal L(S)=o(S^\tau)$, and the ratio in Lemma~\ref{lem:barrier} tends to $+\infty$, a contradiction.

(b) Let $v=-u$. Since $\mathcal A(-\xi)=-\mathcal A(\xi)$, we have $\Delta_p^Hv=-\Delta_p^Hu=e^u=e^{-v}$ weakly in $E$. Apply Lemma~\ref{lem:barrier} with $\mathcal O=\mathbb{R}$ and $h(t)=e^{-t}$. Here $m(r)=-\max_{\{H_0=r\}}u$ and $h(m(r))=\max_{\{H_0=r\}}e^u$. Suppose that the $\limsup$ in (b) is smaller than $p^{p-1}(N-p)$, and choose $\varepsilon<p^{p-1}(N-p)$ and $R_1>R_0$ with $h(m(r))\le\varepsilon r^{-p}$ for $r\ge R_1$. Then $F(r)\le\frac{\varepsilon}{N-p}r^{N-p}$, so
\[
\Lambda(S)\le\rho\log\frac S{R_1},\qquad\rho:=\Big(\frac{\varepsilon}{N-p}\Big)^{\frac1{p-1}}<p .
\]
Also $m(S)\ge p\log S-\log\varepsilon$, and $\mathcal L(S)\le\mathcal L(\infty)<\infty$ because $N>p$. The ratio in Lemma~\ref{lem:barrier} is at least $\big((p-\rho)\log S-C\big)/\mathcal L(\infty)$, which tends to $+\infty$, a contradiction.

By Lemma~\ref{lem:radial}, $U_q$ and $U_e$ solve the two equations in $\mathbb{R}^N\setminus\{0\}$. For the first, $(r^{N-1}\varphi_p((c_sr^\tau)'))'=c_s^{p-1}\tau^{p-1}(N-q\tau)r^{N-1-q\tau}$ and $c_s^{p-1}\tau^{p-1}(N-q\tau)=c_s^{-q}$. They give equality in (a) and (b).
\end{proof}

\section{Solutions stable outside a compact set}\label{sec_outside}

In view of Theorem~\ref{thm:exterior}, Theorem~\ref{thm:outside} follows once we show that a solution stable outside a compact set lies above $U_q$, respectively that $e^u$ lies below $e^{U_e}$, at infinity by an arbitrarily large factor. This is where stability is used. Throughout this section $B(t)=t^p/p$, so $\kappa=0$, $\gamma=\Gamma=1$, $\lambda=p-1$, $\Omega_1=\emptyset$, and $\Omega_2=\mathbb{R}^N$. We fix $R_0>0$ so that $K\subset B_{R_0}$, and we use a cutoff function $\zeta_R\in C^1_c(\mathbb{R}^N\setminus\overline{B_{R_0}})$ equal to $0$ in $B_{R_0+1}$, to $1$ in $B_R\setminus B_{R_0+2}$, and to $0$ outside $B_{2R}$, with $|\nabla\zeta_R|\le C$ in $B_{R_0+2}\setminus B_{R_0+1}$ and $|\nabla\zeta_R|\le C/R$ in $B_{2R}\setminus B_R$, where $C$ does not depend on $R>R_0+3$.

\begin{lemma}\label{lem:decay_sing}
Assume that $(H)$ holds, $p\ge2$, $p\le N<\frac{p(p+3)}{p-1}$, and $q>q_c(N,p)$. Let $u$ be a positive weak solution of $\Delta_p^Hu=u^{-q}$ in $\mathbb{R}^N$ that is stable outside a compact set $K$. Then $u$ is continuous in $\mathbb{R}^N\setminus K$ and
\[
\lim_{|x|\to\infty}|x|^{-\tau}u(x)=+\infty .
\]
\end{lemma}

\begin{proof}
Set $\beta:=\frac{N(q+p-1)}p$, so that $\frac N\beta=\tau$, and $\alpha_*:=q-\beta$. Then $\alpha_*\le1-p$ because $N\ge p$. By Corollary~\ref{cor:stable}(b), the assumption $q>q_c(N,p)$ means that $\alpha_0<1-p$ and $N(q+p-1)<p(q-\alpha_0)$, so $\alpha_*>\alpha_0$. For $\varepsilon\in(0,p)$ put $s_\varepsilon:=\frac N{p-\varepsilon}$ and $\alpha_\varepsilon:=q-s_\varepsilon(q+p-1)$. Then $\alpha_\varepsilon<\alpha_*$ and $\alpha_\varepsilon\to\alpha_*$ as $\varepsilon\to0$, and we fix $\varepsilon$ so small that $\alpha_\varepsilon>\alpha_0$.

\medskip\noindent\textit{Step 1: continuity.} Write the equation as $\Delta_p^Hu=d(x)u^{p-1}$ with $d:=u^{-(q+p-1)}$, so that $d^{s_\varepsilon}=u^{\alpha_\varepsilon-q}$. For $\bar x\notin K$ and a ball $B_{2\rho}(\bar x)\subset\mathbb{R}^N\setminus K$, Proposition~\ref{prop:est_sing} with $\alpha=\alpha_\varepsilon$ and a cutoff function supported in $B_{2\rho}(\bar x)$ shows that $d\in L^{s_\varepsilon}(B_\rho(\bar x))$. Since $s_\varepsilon>\frac Np$, Lemma~\ref{lem:serrin}(i) shows that $u$ is locally bounded and continuous in $\mathbb{R}^N\setminus K$, and Lemma~\ref{lem:serrin}(ii) shows that $u$ is locally bounded away from $0$ there.

\medskip\noindent\textit{Step 2: an integral bound at infinity.} Proposition~\ref{prop:est_sing} with $\alpha=\alpha_*$ and $\eta=\zeta_R$ gives, for $R>R_0+3$,
\[
\int_{B_R\setminus B_{R_0+2}}u^{-\beta}\le C\int_{B_{R_0+2}\setminus B_{R_0+1}}|\nabla\zeta_R|^{\frac{p(q-\alpha_*)}{q+p-1}}+C\int_{B_{2R}\setminus B_R}|\nabla\zeta_R|^{\frac{p(q-\alpha_*)}{q+p-1}}\le C+CR^{N-\frac{p\beta}{q+p-1}}=2C,
\]
since $\frac{p\beta}{q+p-1}=N$. Hence $u^{-\beta}\in L^1(\mathbb{R}^N\setminus B_{R_0+2})$. If $p=N=2$, then $\alpha_*=-1$, which is allowed in Proposition~\ref{prop:est_sing}.

\medskip\noindent\textit{Step 3: a Harnack inequality at infinity.} Let $R\ge R_0+2$ and $|y|=8R$. Then $B_{4R}(y)\subset\{4R<|x|<12R\}$, so $u$ is stable in $B_{4R}(y)$. Proposition~\ref{prop:est_sing} with $\eta(x)=\eta_{2R}(x-y)$, which equals $1$ on $B_{2R}(y)$ and satisfies $|\nabla\eta|\le1/R$, gives
\[
\int_{B_{2R}(y)}u^{\alpha_\varepsilon-q}\le CR^{N-\frac{p(q-\alpha_\varepsilon)}{q+p-1}}=CR^{N-ps_\varepsilon},
\]
with $C$ independent of $R$ and $y$. Therefore
\[
R^{p-\frac N{s_\varepsilon}}\|d\|_{L^{s_\varepsilon}(B_{2R}(y))}=R^{\varepsilon}\Big(\int_{B_{2R}(y)}u^{\alpha_\varepsilon-q}\Big)^{\frac{p-\varepsilon}N}\le C,
\]
and Lemma~\ref{lem:serrin}(ii) gives $\sup_{B_R(y)}u\le K_*\inf_{B_R(y)}u$ with $K_*$ independent of $R$ and $y$.

\medskip\noindent\textit{Step 4: conclusion.} Let $\vartheta>0$. By Step~2 there is $R_\vartheta\ge R_0+2$ with $\int_{\{|x|>R_\vartheta\}}u^{-\beta}\le\vartheta$. For $R\ge R_\vartheta$ and $|y|=8R$ we have $B_R(y)\subset\{|x|>7R\}$, so
\[
|B_1|R^N\Big(\sup_{B_R(y)}u\Big)^{-\beta}\le\int_{B_R(y)}u^{-\beta}\le\vartheta,\qquad\text{hence}\qquad\sup_{B_R(y)}u\ge\Big(\frac{|B_1|}\vartheta\Big)^{\frac1\beta}R^\tau .
\]
By Step~3, $u(y)\ge K_*^{-1}(|B_1|/\vartheta)^{1/\beta}(|y|/8)^\tau$ for $|y|\ge8R_\vartheta$. Since $\vartheta$ is arbitrary, the lemma follows.
\end{proof}

\begin{lemma}\label{lem:decay_exp}
Assume that $(H)$ holds, $p\ge2$, and $p<N<\frac{p(p+3)}{p-1}$. Let $u$ be a weak solution of $-\Delta_p^Hu=e^u$ in $\mathbb{R}^N$ that is stable outside a compact set $K$. Then $u$ is continuous in $\mathbb{R}^N\setminus K$ and
\[
\lim_{|x|\to\infty}|x|^pe^{u(x)}=0 .
\]
\end{lemma}

\begin{proof}
Put $\alpha_*:=\frac Np-1$. The assumption on $N$ means exactly that $0<\alpha_*<\frac4{p-1}=\frac4\lambda$. Fix $\varepsilon\in(0,p)$ so small that $\alpha_\varepsilon:=\frac N{p-\varepsilon}-1<\frac4{p-1}$, and put $s_\varepsilon:=\alpha_\varepsilon+1=\frac N{p-\varepsilon}$.

\medskip\noindent\textit{Step 1: continuity.} By Proposition~\ref{prop:est_exp} with $\alpha=\alpha_\varepsilon$ and cutoff functions supported in small balls, $e^u\in L^{s_\varepsilon}_{\rm loc}(\mathbb{R}^N\setminus K)$. Since $s_\varepsilon>\frac Np$, Lemma~\ref{lem:serrin}(i) shows that $u$ is locally bounded and continuous in $\mathbb{R}^N\setminus K$.

\medskip\noindent\textit{Step 2: an integral bound at infinity.} Proposition~\ref{prop:est_exp} with $\alpha=\alpha_*$, $m=\alpha_*+1$, and $\eta=\zeta_R$ gives
\[
\int_{B_R\setminus B_{R_0+2}}e^{\frac Npu}\le C+CR^{N-p(\alpha_*+1)}=2C,
\]
so $e^{\frac Npu}\in L^1(\mathbb{R}^N\setminus B_{R_0+2})$.

\medskip\noindent\textit{Step 3: a local maximum principle at infinity.} Let $\theta>0$ and $w:=e^{\theta u}$. Since $u$ is locally bounded in $\mathbb{R}^N\setminus K$, $w\in W^{1,p}_{\rm loc}(\mathbb{R}^N\setminus K)$ and $\nabla w=\theta w\nabla u$. By the homogeneity of $\mathcal A$, $\mathcal A(\nabla w)=\theta^{p-1}w^{p-1}\mathcal A(\nabla u)$. For $0\le\varphi\in C^1_c(\mathbb{R}^N\setminus K)$, the function $w^{p-1}\varphi$ is bounded, belongs to $W^{1,p}$, and has compact support, so it is admissible in the equation, and $\langle\mathcal A(\nabla u),\nabla u\rangle=H(\nabla u)^p\ge0$. Hence
\[
\int\langle\mathcal A(\nabla w),\nabla\varphi\rangle=\theta^{p-1}\int\langle\mathcal A(\nabla u),\nabla(w^{p-1}\varphi)\rangle-(p-1)\theta^p\int w^{p-1}H(\nabla u)^p\varphi\le\theta^{p-1}\int e^uw^{p-1}\varphi .
\]
That is, $-\Delta_p^Hw\le a(x)w^{p-1}$ with $a=\theta^{p-1}e^u$. Now take $\theta:=\frac N{p^2}$, so that $w^p=e^{\frac Npu}$. Let $R\ge R_0+2$ and $|y|=8R$. As in the proof of Lemma~\ref{lem:decay_sing}, Proposition~\ref{prop:est_exp} with $\alpha=\alpha_\varepsilon$ and $\eta(x)=\eta_{2R}(x-y)$ gives $\int_{B_{2R}(y)}e^{s_\varepsilon u}\le CR^{N-ps_\varepsilon}$, hence
\[
R^{p-\frac N{s_\varepsilon}}\|a\|_{L^{s_\varepsilon}(B_{2R}(y))}\le CR^{\varepsilon}\big(R^{N-ps_\varepsilon}\big)^{\frac{p-\varepsilon}N}=C .
\]
Lemma~\ref{lem:serrin}(iii) gives $\sup_{B_R(y)}w\le K_*R^{-\frac Np}\|w\|_{L^p(B_{2R}(y))}$ with $K_*$ independent of $R$ and $y$. Since $w^{1/\theta}=e^u$, this reads
\begin{equation}\label{eq:locmax}
\sup_{B_R(y)}e^u\le K_*^{\frac1\theta}R^{-p}\Big(\int_{B_{2R}(y)}e^{\frac Npu}\Big)^{\frac pN}.
\end{equation}

\medskip\noindent\textit{Step 4: conclusion.} Let $\vartheta>0$ and choose $R_\vartheta\ge R_0+2$ with $\int_{\{|x|>R_\vartheta\}}e^{\frac Npu}\le\vartheta$. For $|y|\ge8R_\vartheta$ let $R=|y|/8$. Then $B_{2R}(y)\subset\{|x|>6R\}$, and \eqref{eq:locmax} gives $|y|^pe^{u(y)}\le8^pK_*^{1/\theta}\vartheta^{p/N}$. Since $\vartheta$ is arbitrary, the lemma follows.
\end{proof}

In \cite{HL26b} the argument of Step~3 was applied to $e^u$ without knowing that $e^u\in W^{1,p}_{\rm loc}$. That fact follows from Step~1, which uses the full range $N<\frac{p(p+3)}{p-1}$.

\begin{proof}[Proof of Theorem~\ref{thm:outside}]
(a) Let $u$ be a positive weak solution of $\Delta_p^Hu=u^{-q}$ that is stable outside a compact set $K$, and choose $R_0'$ with $K\subset\{H_0\le R_0'\}$. By Lemma~\ref{lem:decay_sing}, $u$ is continuous in $\{H_0>R_0'\}$, and since $H_0$ is equivalent to the Euclidean norm, $r^{-\tau}\min_{\{H_0=r\}}u\to+\infty$ as $r\to\infty$. This contradicts Theorem~\ref{thm:exterior}(a), applied with $R_0'$ in place of $R_0$.

(b) Similarly, Lemma~\ref{lem:decay_exp} gives $r^p\max_{\{H_0=r\}}e^u\to0$, which contradicts Theorem~\ref{thm:exterior}(b), since $N>p$.

The statement on finite Morse index follows from the observation after the definition of stability.
\end{proof}

\begin{remark}\label{rem:HL2}
Lemma~\ref{lem:decay_sing} follows \cite[proof of Proposition~3.2]{GM16}. It gives more than the lower bound $u\ge K_1|x|^\tau$ obtained there, and this is needed: a lower bound with a constant $K_1\le c_s$ would not contradict Theorem~\ref{thm:exterior}(a). Theorem~\ref{thm:outside}(b) has the same statement as \cite[Theorem~3(i)]{HL26b}, while \cite[Theorem~4]{HL26a} cannot hold as stated when $N=p>2$ and $q\le\frac{(p-2)^2}4$, by Theorem~\ref{thm:sharp}(a).
\end{remark}

\section{The case \texorpdfstring{$N=p$}{N=p}: classification}\label{sec_N}

The proof of Theorem~\ref{thm:Nliouville} has three steps. We first obtain the scale-invariant pointwise decay $e^u\le C|x|^{-N}$. We then combine annular energy estimates with a weighted logarithmic cutoff to prove that $e^u$ has finite mass. The classification theorem of Ciraolo and Li then identifies the solution, while a Hardy-type argument proves the converse stability statement.

In this section $B(t)=t^N/N$, so $\kappa=0$, $\Omega_1=\emptyset$, and $\mathcal M(\xi)=D^2(H^N/N)(\xi)$ for $\xi\ne0$. By \eqref{eq:homog} and \eqref{eq:Mbound}, for $\xi\neq0$,
\begin{equation}\label{eq:MN}
\langle\mathcal M(\xi)\xi,\xi\rangle=(N-1)H(\xi)^N,\qquad |\mathcal M(\xi)|\le\Lambda_1H(\xi)^{N-2},
\end{equation}
and for $N=2$ we choose $\Lambda_1$ so large that also $|\mathcal M(0)|\le\Lambda_1$. When $N=2$, expressions involving $H(\nabla u)^{N-2}$ are understood with the convention $H(\nabla u)^0=1$. The equation $-\Delta_N^Hu=e^u$ is invariant under $u\mapsto u(R\,\cdot)+N\log R$. If $u\in W^{1,N}_{\rm loc}(\mathbb{R}^N)$ is a weak solution, then $e^{|u|}\in L^s_{\rm loc}$ for every $s<\infty$ by the Trudinger--Moser inequality \cite{Tru67b}, and Lemma~\ref{lem:serrin}(i) shows that $u$ is locally bounded and continuous. In particular the condition $e^u\in L^1_{\rm loc}$ in the definition of a weak solution is automatic, and bounded test functions such as $e^{\theta u}\varphi$ are admissible.

\begin{lemma}\label{lem:decay_N}
Assume $(H)$ and $N\ge2$, and let $u$ be a weak solution of $-\Delta_N^Hu=e^u$ that is stable outside a compact set $K\subset B_{R_0}$. There is $C_0$ such that $e^{u(x)}\le C_0|x|^{-N}$ for $|x|\ge8(R_0+2)$.
\end{lemma}

\begin{proof}
Fix $\alpha\in(0,\frac4{N-1})$, and let $R\ge R_0+2$ and $|y|=8R$. As in Step~3 of the proof of Lemma~\ref{lem:decay_exp}, Proposition~\ref{prop:est_exp} with $p=N$ and $\eta(x)=\eta_{2R}(x-y)$ gives
\begin{equation}\label{eq:locintN}
\int_{B_{2R}(y)}e^{(1+\alpha)u}\le CR^{N-N(1+\alpha)}=CR^{-N\alpha},
\end{equation}
with $C$ independent of $R$ and $y$. Let $\theta:=\frac{1+\alpha}N$ and $w:=e^{\theta u}$. The computation in Step~3 of the proof of Lemma~\ref{lem:decay_exp}, which is valid for every $\theta>0$, shows that $-\Delta_N^Hw\le a(x)w^{N-1}$ with $a=\theta^{N-1}e^u$. With $s=1+\alpha>1$, \eqref{eq:locintN} gives $R^{N-\frac Ns}\|a\|_{L^s(B_{2R}(y))}\le CR^{N-\frac Ns}R^{-\frac{N\alpha}s}=C$, and Lemma~\ref{lem:serrin}(iii) gives
\[
\sup_{B_R(y)}e^{\theta u}\le K_*R^{-1}\Big(\int_{B_{2R}(y)}e^{(1+\alpha)u}\Big)^{\frac1N}\le CR^{-1-\alpha}.
\]
Raising to the power $\frac1\theta=\frac N{1+\alpha}$ gives $e^{u(y)}\le CR^{-N}=8^NC|y|^{-N}$.
\end{proof}

Let $u$ be as in Lemma~\ref{lem:decay_N}, and fix an integer $j_0$ with $2^{j_0-2}\ge8(R_0+2)$. Consider the annuli
\[
\widetilde A:=\{\tfrac14<|z|<8\},\qquad A':=\{\tfrac12<|z|<4\},\qquad A'':=\{1<|z|<2\},\qquad A_j:=\{2^j<|x|<2^{j+1}\},
\]
and, for $j\ge j_0$, the rescaled functions $u_j(z):=u(2^jz)+Nj\log2$, $z\in\widetilde A$. Each $u_j$ solves the equation in $\widetilde A$, and by Lemma~\ref{lem:decay_N}, $e^{u_j(z)}\le C_0|z|^{-N}\le4^NC_0$ there. Put $L:=\log(4^NC_0)$, so that $u_j\le L$ in $\widetilde A$, and
\[
s_j:=L-\sup_{A'}u_j\ \ge0 .
\]
This number measures how far $e^u$ stays below the critical size $|x|^{-N}$ near $A_j$.

\begin{lemma}\label{lem:annulus}
There is $A_*$, independent of $j\ge j_0$, such that
\[
\mu_j:=\int_{A_j}e^u\le A_*e^{-s_j}\qquad\text{and}\qquad W_j:=\int_{A_j}\frac{H(\nabla u)^{N-2}}{|x|^2}\le A_*(1+s_j)^{N-2}.
\]
\end{lemma}

\begin{proof}
Both quantities are invariant under the scaling, so $\mu_j=\int_{A''}e^{u_j}$ and $W_j=\int_{A''}H(\nabla u_j)^{N-2}|z|^{-2}$. The first bound follows from $u_j\le L-s_j$ on $A''\subset A'$.

For the second, let $Z_j:=L-u_j\ge0$ in $\widetilde A$. Since $\mathcal A$ is odd, $-\Delta_N^HZ_j=\Delta_N^Hu_j=-e^{u_j}$, and $e^{u_j}\le e^L$. Covering $A'$ by a fixed finite chain of balls $B_\rho(y)$ with $B_{2\rho}(y)\subset\widetilde A$ and applying Lemma~\ref{lem:serrin}(iv), we get $\sup_{A'}Z_j\le K_*(\inf_{A'}Z_j+1)=K_*(1+s_j)$ with $K_*$ independent of $j$. Hence $\operatorname{osc}_{A'}u_j\le K_*(1+s_j)$. Let $k:=\inf_{A'}u_j$ and let $\zeta\in C^1_c(A')$ satisfy $0\le\zeta\le1$, $\zeta=1$ on $A''$, and $|\nabla\zeta|\le8$. Testing the equation of $u_j$ with $\zeta^N(u_j-k)\ge0$, we get
\[
\int\zeta^NH(\nabla u_j)^N=\int e^{u_j}\zeta^N(u_j-k)-N\int\zeta^{N-1}(u_j-k)\langle\mathcal A(\nabla u_j),\nabla\zeta\rangle .
\]
Since $|\mathcal A(\xi)|\le MH(\xi)^{N-1}$, Young's inequality gives $\int_{A''}H(\nabla u_j)^N\le C\int_{A'}(u_j-k)^N+Ce^L\int_{A'}(u_j-k)\le C(1+s_j)^N$. H\"older's inequality and $|z|\ge1$ on $A''$ give $W_j\le|A''|^{\frac2N}\big(\int_{A''}H(\nabla u_j)^N\big)^{\frac{N-2}N}\le C(1+s_j)^{N-2}$.
\end{proof}

\begin{proposition}\label{prop:mass}
Assume $(H)$ and $N\ge2$. If $u$ is a weak solution of $-\Delta_N^Hu=e^u$ that is stable outside a compact set, then $\int_{\mathbb{R}^N}e^u<\infty$.
\end{proposition}

\begin{proof}
Put $w_j:=(1+s_j)^{N-2}$. For $N=2$, $w_j=1$. Fix integers $j_0<J<J'$, and let $\eta(x)=\omega(\log_2|x|)$, where $\omega$ is continuous and piecewise linear, $\omega=0$ on $(-\infty,j_0-1]$, $\omega$ increases linearly to $1$ on $[j_0-1,j_0]$, $\omega=1$ on $[j_0,J]$, $\omega=0$ on $[J',\infty)$, and on $[j,j+1]$, $J\le j<J'$, $\omega$ has slope $-c_j$ with $c_j\ge0$ and $\sum_{j=J}^{J'-1}c_j=1$. Then $\eta$ is Lipschitz with compact support in $\{|x|>2^{j_0-1}\}\subset\mathbb{R}^N\setminus K$, and $|\nabla\eta|\le c_j/(|x|\log2)$ on $A_j$ for $J\le j<J'$. The stability inequality, which holds for such $\eta$ by the density argument of Step~1 in the proof of Lemma~\ref{lem:key}, together with \eqref{eq:MN} and Lemma~\ref{lem:annulus}, gives
\[
\sum_{j=j_0}^{J-1}\mu_j\le\int e^u\eta^2\le\Lambda_1\int H(\nabla u)^{N-2}|\nabla\eta|^2\le D_0+C\sum_{j=J}^{J'-1}c_j^2w_j ,
\]
where
\[
D_0:=\Lambda_1\int_{A_{j_0-1}}H(\nabla u)^{N-2}|\nabla\eta|^2
\]
is finite and independent of $J$ and $J'$. The choice $c_j=w_j^{-1}\big(\sum_{i=J}^{J'-1}w_i^{-1}\big)^{-1}$ minimizes the last sum, and we get
\[
\sum_{j=j_0}^{J-1}\mu_j\le D_0+C\Big(\sum_{j=J}^{J'-1}w_j^{-1}\Big)^{-1}\qquad(j_0<J<J').
\]
If $\sum_jw_j^{-1}=\infty$, we let $J'\to\infty$ and obtain $\sum_{j\ge j_0}\mu_j\le D_0$. If $\sum_jw_j^{-1}<\infty$, Lemma~\ref{lem:annulus} and the inequality $e^{-s}\le C_N(1+s)^{-(N-2)}$ for $s\ge0$ give $\sum_{j\ge j_0}\mu_j\le A_*C_N\sum_jw_j^{-1}<\infty$. In both cases $\int_{\{|x|>2^{j_0}\}}e^u<\infty$, and $e^u$ is locally integrable.
\end{proof}

For $N=2$ only the first alternative occurs, and the argument is the logarithmic cutoff argument of \cite{Far07,DF09}. The integral estimate alone does not give finite mass when $N\ge3$: by \eqref{eq:locintN}, each dyadic annulus carries mass at most $C$, which is the scale-invariant bound and is not summable. For $N\ge4$ the functions \eqref{eq:bubbles} themselves fall into the second alternative, since $e^u$ decays like $|x|^{-N^2/(N-1)}$ and $s_j$ grows linearly in $j$.

\begin{proof}[Proof of Theorem~\ref{thm:Nliouville}]
If $u$ is stable outside a compact set, then $e^u\in L^1(\mathbb{R}^N)$ by Proposition~\ref{prop:mass}, and $u$ is one of the functions \eqref{eq:bubbles} by \cite[Theorem~1.1]{CL24}.

Conversely, let $u=U_{\ell,x_0}$. By translation invariance we may assume $x_0=0$. Put $r=H_0(x)$ and $\beta=\frac N{N-1}$. By Lemma~\ref{lem:dual},
\[
H(\nabla u(x))=\frac{N\beta\ell^\beta r^{\beta-1}}{1+\ell^\beta r^\beta}\ \ge\ \frac c{|x|}\qquad\text{and}\qquad e^{u(x)}\le C|x|^{-N\beta}\qquad(|x|\ge1).
\]
Since $D^2(H^2/2)\ge\lambda_0I$ and $\mathcal M(\xi)=H(\xi)^{N-2}\big[(N-2)\nabla H(\xi)\otimes\nabla H(\xi)+D^2(H^2/2)(\xi)\big]$, we have $\langle\mathcal M(\xi)\eta,\eta\rangle\ge\lambda_0H(\xi)^{N-2}|\eta|^2$ for $\xi\neq0$ and $\eta\in\mathbb{R}^N$. Choose $R_1>1$ such that
\[
C|x|^{-N\beta}\le\frac{\lambda_0c^{N-2}}4\,|x|^{-N}(\log|x|)^{-2}\qquad\text{for }|x|\ge R_1 .
\]
The function $v=(\log|x|)^{1/2}$ satisfies $-\operatorname{div}(|x|^{2-N}\nabla v)=\frac{v}{4|x|^N(\log|x|)^2}$ in $\{|x|>1\}$, so for $\varphi=v\psi\in C^1_c(\{|x|>R_1\})$,
\[
\int\frac{|\nabla\varphi|^2}{|x|^{N-2}}=\int\frac{\varphi^2}{4|x|^N(\log|x|)^2}+\int\frac{v^2|\nabla\psi|^2}{|x|^{N-2}}\ \ge\ \int\frac{\varphi^2}{4|x|^N(\log|x|)^2}.
\]
Hence $Q_u(\varphi)\ge\lambda_0c^{N-2}\int|x|^{2-N}|\nabla\varphi|^2\ge\int e^u\varphi^2$, and $u$ is stable outside $\overline{B_{R_1}}$.
\end{proof}

\section{Sharpness}\label{sec_sharp}

In this section $B(t)=t^p/p$, so for $\xi\neq0$
\[
\mathcal M(\xi)=D^2(H^p/p)(\xi)=(p-1)H(\xi)^{p-2}\nabla H(\xi)\otimes\nabla H(\xi)+H(\xi)^{p-1}D^2H(\xi).
\]
 This matrix is symmetric and positive semidefinite since $H^p$ is convex. All three parts of Theorem~\ref{thm:sharp} rest on the same transfer of stability from a radial function to the anisotropic problem.

\begin{lemma}\label{lem:picone}
Assume that $(H)$ holds and $p\ge2$. Let $\Omega\subset\mathbb{R}^N$ be open, let $u\in C^1(\Omega)$ with $\nabla u\ne0$ in $\Omega$, and let $0\le V\in L^1_{\rm loc}(\Omega)$. Suppose that $W\in C^1(\Omega)$ is positive and satisfies
\[
\int_\Omega\langle\mathcal M(\nabla u)\nabla W,\nabla\zeta\rangle\ge\int_\Omega VW\zeta
\]
for every nonnegative Lipschitz function $\zeta$ with compact support in $\Omega$. Then $\int_\Omega\langle\mathcal M(\nabla u)\nabla\varphi,\nabla\varphi\rangle\ge\int_\Omega V\varphi^2$ for every $\varphi\in C^1_c(\Omega)$.
\end{lemma}

\begin{proof}
Write $\mathcal M=\mathcal M(\nabla u)$, let $\psi:=\varphi/W\in C^1_c(\Omega)$, and take $\zeta=W\psi^2$. Since $\nabla\varphi=\psi\nabla W+W\nabla\psi$ and $\mathcal M$ is symmetric,
\[
\langle\mathcal M\nabla\varphi,\nabla\varphi\rangle=\langle\mathcal M\nabla W,\nabla(W\psi^2)\rangle+W^2\langle\mathcal M\nabla\psi,\nabla\psi\rangle .
\]
Integrating, and using the hypothesis and $\mathcal M\ge0$, we obtain the claim.
\end{proof}

For part (a) we need the sign of $\tau v-rv'$ along radial solutions. This was proved in \cite[proof of Theorem~2.10]{GM16} for $p<N$. In that proof the lower bound \eqref{eq:Phi} below is obtained by multiplying linear lower bounds for the two factors of $\Phi$. The product of these bounds contains a negative quadratic term, which is discarded, and for $p>2$ this does not give a lower bound. We therefore give the argument in full. It follows \cite{GM16}, except that \eqref{eq:Phi} is derived from convexity, and it also covers $N=p$.

\begin{lemma}\label{lem:z}
Let $p\ge2$, $q>0$, and $N\ge2$, and assume that
\begin{equation}\label{eq:zcond}
N-(q+1)\tau>0\qquad\text{and}\qquad D:=\big[N-(q-1)\tau\big]^2-\frac{4p}{p-1}(N-q\tau)\ge0 .
\end{equation}
Let $a>0$, and let $v$ be the solution of $(r^{N-1}\varphi_p(v'))'=r^{N-1}v^{-q}$ on $(0,\infty)$ with $v(0)=a$ and $v'(0)=0$. Then $z(r):=\tau v(r)-rv'(r)>0$ for all $r\ge0$, and $v(r)>c_sr^\tau$ for all $r>0$.
\end{lemma}

\begin{proof}
Integrating the equation, $v'>0$ on $(0,\infty)$ and $v'(r)=(a^{-q}r/N)^{\frac1{p-1}}(1+o(1))$ as $r\to0$. Put $s=\log r$,
\[
X(s):=r^{-\tau}v(r),\qquad Y(s):=\big(r^{1-\tau}v'(r)\big)^{p-1}>0 .
\]
Using $\tau(q+p-1)=p$, a direct computation gives
\[
X'=Y^{\frac1{p-1}}-\tau X,\qquad Y'=X^{-q}-(N-q\tau)Y,
\]
and $z=-r^\tau X'$. Note that $N-q\tau>N-(q+1)\tau>0$. The only equilibrium is $(L,(\tau L)^{p-1})$, where $L=c_s$, that is, $L^{-q}=(N-q\tau)(\tau L)^{p-1}$. As $s\to-\infty$, $X\to+\infty$, and $Y\to0$ because $\tau<\frac p{p-1}$. Moreover $X'+\tau X=Y^{1/(p-1)}\to0$.

\medskip\noindent\textit{Step 1: $X>L$ on $\mathbb{R}$.} Suppose not, and let $s_0$ be the first point with $X(s_0)=L$. On $(-\infty,s_0)$ we have $X>L$, hence $\big(e^{(N-q\tau)s}Y\big)'<e^{(N-q\tau)s}L^{-q}$. Integrating from $-\infty$ gives $Y<(\tau L)^{p-1}$ on $(-\infty,s_0]$, and therefore $X'=Y^{1/(p-1)}-\tau X<0$ there, including at $s_0$. So $X$ decreases strictly from $+\infty$ to $L$ on $(-\infty,s_0]$, and we can write $X'=h(X)$ with $h\in C^1([L,\infty))$ and $h<0$. Differentiating $X'=Y^{1/(p-1)}-\tau X$ and using the equation for $Y$, we find that
\[
h(\xi)h'(\xi)=-\tau h(\xi)+\frac1{p-1}\big(h(\xi)+\tau\xi\big)^{2-p}\xi^{-q}-\frac{N-q\tau}{p-1}\big(h(\xi)+\tau\xi\big),\qquad \xi\ge L,
\]
where $h(\xi)+\tau\xi=Y^{1/(p-1)}>0$, and $h(\xi)+\tau\xi\to0$ as $\xi\to\infty$.

Fix $\mu>\tau$. The function $\xi\mapsto h(\xi)-\mu(L-\xi)$ is negative at $\xi=L$ and equals $(\mu-\tau)\xi-\mu L+o(1)$ as $\xi\to\infty$. Let $\xi_\mu>L$ be its smallest zero. Then $h'(\xi_\mu)\ge-\mu$, and with $\sigma:=\xi_\mu-L>0$ we have $h(\xi_\mu)=-\mu\sigma$ and $h(\xi_\mu)+\tau\xi_\mu=\tau L-(\mu-\tau)\sigma>0$. Put
\[
\Phi(t):=\big(\tau L-(\mu-\tau)t\big)^{2-p}(L+t)^{-q},\qquad 0\le t<\frac{\tau L}{\mu-\tau}.
\]
Both $\log(\tau L-(\mu-\tau)t)^{2-p}$ and $\log(L+t)^{-q}$ are convex, the second strictly, so $\Phi$ is strictly convex and
\begin{equation}\label{eq:Phi}
\Phi(\sigma)>\Phi(0)+\Phi'(0)\sigma=(N-q\tau)\tau L+(N-q\tau)\big[(p-2)(\mu-\tau)-q\tau\big]\sigma,
\end{equation}
where we used $L^{-q}=(N-q\tau)(\tau L)^{p-1}$. Dividing the equation for $h$ at $\xi_\mu$ by $h(\xi_\mu)=-\mu\sigma<0$ and using \eqref{eq:Phi}, we obtain
\[
h'(\xi_\mu)=-\tau-\frac{\Phi(\sigma)}{(p-1)\mu\sigma}+\frac{(N-q\tau)\big(\tau L-(\mu-\tau)\sigma\big)}{(p-1)\mu\sigma}<-\tau+\frac{N-q\tau}{(p-1)\mu}\big[q\tau-(p-1)(\mu-\tau)\big].
\]
Combining this with $h'(\xi_\mu)\ge-\mu$, multiplying by $(p-1)\mu$, and using $\tau(q+p-1)=p$, we get
\begin{equation}\label{eq:quad}
\mu^2-\big[N-(q-1)\tau\big]\mu+\frac{p}{p-1}(N-q\tau)>0\qquad\text{for every }\mu>\tau .
\end{equation}
By \eqref{eq:zcond}, the left-hand side has real roots $\mu_-\le\mu_+$. Since $N-(q+1)\tau>0$, the inequality $\mu_->\tau$ is equivalent to $\big(N-(q+1)\tau\big)^2>D$, that is, to $-4\tau(N-q\tau)>-\frac{4p}{p-1}(N-q\tau)$, which holds because $\tau<\frac p{p-1}$. Taking $\mu=\mu_-$ in \eqref{eq:quad} gives a contradiction. Hence $X>L$ on $\mathbb{R}$.

\medskip\noindent\textit{Step 2: $X'<0$ on $\mathbb{R}$.} Since $X>L$ everywhere, the integration in Step~1 gives $Y<(\tau L)^{p-1}$ on $\mathbb{R}$, so $X'=Y^{1/(p-1)}-\tau X<\tau L-\tau L=0$. Thus $z=-r^\tau X'>0$ for $r>0$, and $z(0)=\tau a>0$.
\end{proof}

The conditions \eqref{eq:zcond} hold in the cases of Theorem~\ref{thm:sharp}(a) with $q\le q_c(N,p)$. If $N>p$, then $(q+1)\tau\le p<N$ since $p\ge2$, and, with $m:=N-p>0$, we have $N-q\tau=m+(p-1)\tau$, $N-(q-1)\tau=m+p\tau$, and hence $D=p^2\tau^2+2p(m-2)\tau+m\big(m-\frac{4p}{p-1}\big)$. If $m\ge\frac{4p}{p-1}$, all coefficients are nonnegative and $D>0$. Otherwise the constant term is negative, so $D$ has exactly one positive root in $\tau$, and $D\ge0$ if and only if $\tau$ is at least this root. Since $\tau$ decreases in $q$, this means $q\le q_c(N,p)$ by \eqref{eq:qc}. If $N=p>2$, then $N-q\tau=\tau(p-1)$, $N-(q+1)\tau=\tau(p-2)>0$, and $D=p\tau(p\tau-4)$, which is nonnegative if and only if $q\le\frac{(p-2)^2}4$.

\begin{proof}[Proof of Theorem~\ref{thm:sharp}(a)]
Let $v=v_a$ be the solution of $(r^{N-1}\varphi_p(v'))'=r^{N-1}v^{-q}$ on $(0,\infty)$ with $v(0)=a$ and $v'(0)=0$. It exists for all $r>0$, is positive and unique \cite[Theorem~2.1]{GM16}, and $v'>0$ on $(0,\infty)$. By Lemma~\ref{lem:radial}, $u(x)=v(H_0(x))$ solves $\Delta_p^Hu=u^{-q}$ in $\mathbb{R}^N\setminus\{0\}$. Integrating the ODE from $0$ to $r$ and using $v(r)\to a$ gives $v'(r)=(a^{-q}r/N)^{1/(p-1)}(1+o(1))$ as $r\to0$. Since $\nabla H_0$ is bounded on $\mathbb{R}^N\setminus\{0\}$ by homogeneity, $\nabla u(x)\to0$ as $x\to0$, so $u\in C^1(\mathbb{R}^N)$. Moreover, $\mathcal A(\nabla u)=\varphi_p(v'(r))x/r$ extends continuously by $0$ at the origin. Removing a small Wulff ball and letting its radius tend to zero therefore shows that $u$ is a weak solution in all of $\mathbb{R}^N$. Conversely, let $u=v(H_0(x))$ be a positive weak solution of class $C^1$. Then $v'(0)=0$, because $H_0$ is even and $u$ is differentiable at the origin. Testing the equation with $\psi(H_0(x))$, $\psi\in C^1_c((0,\infty))$, and using Lemma~\ref{lem:dual} and the coarea formula for $H_0$, we get $(r^{N-1}\varphi_p(v'))'=r^{N-1}v^{-q}$ in the sense of distributions on $(0,\infty)$. Hence $r^{N-1}\varphi_p(v')\in C^1((0,\infty))$ and $v$ solves the initial value problem, which proves uniqueness.

Let $z(r):=\tau v(r)-rv'(r)$. Since $\omega^{-\tau}v(\omega r)$ is a solution of the same ordinary differential equation for every $\omega>0$, differentiation in $\omega$ at $\omega=1$ gives
\begin{equation}\label{eq:lin}
(p-1)\big(r^{N-1}|v'|^{p-2}z'\big)'=-q\,r^{N-1}v^{-q-1}z\quad\text{on }(0,\infty),\qquad r^{N-1}|v'|^{p-2}z'(r)\to0\ \text{as }r\to0 .
\end{equation}
By Lemma~\ref{lem:z} and the remark after it, $z>0$ on $[0,\infty)$. Since $z\in C^1([0,\infty))$ with $z'(0)=0$, the function $W(x):=z(H_0(x))$ is of class $C^1$ and positive.

Since $\nabla W=z'(r)\nabla H_0$ is parallel to $\nabla u=v'(r)\nabla H_0$, \eqref{eq:homog} and Lemma~\ref{lem:dual} give $D^2H(\nabla u)\nabla W=0$, $\langle\nabla H(\nabla u),\nabla W\rangle=z'(r)$, and
\[
\mathcal M(\nabla u)\nabla W=(p-1)|v'(r)|^{p-2}z'(r)\,\frac{x}{r} .
\]
The function $G:=(p-1)|v'|^{p-2}z'$ satisfies $r^{N-1}G\in C^1((0,\infty))$ by \eqref{eq:lin}, so \eqref{eq:div} and \eqref{eq:lin} give $-\operatorname{div}(\mathcal M(\nabla u)\nabla W)=qu^{-q-1}W$ in $\mathbb{R}^N\setminus\{0\}$, classically and weakly. Lemma~\ref{lem:picone} with $\Omega=\mathbb{R}^N\setminus\{0\}$ and $V=qu^{-q-1}$ gives the stability inequality for $\varphi\in C^1_c(\mathbb{R}^N\setminus\{0\})$. A general $\varphi\in C^1_c(\mathbb{R}^N)$ is the limit of $\varphi_\rho:=\varphi(1-\eta_\rho)$ as $\rho\to0$. Since $|\mathcal M(\nabla u)|$ and $u^{-q-1}$ are bounded near $0$, and $N\ge3$, we have $\int|\mathcal M(\nabla u)|\varphi^2|\nabla\eta_\rho|^2\le C\rho^{N-2}\to0$. The mixed term in $Q_u(\varphi_\rho)$ tends to zero by the Cauchy--Schwarz inequality for the positive semidefinite matrix $\mathcal M(\nabla u)$, while the remaining terms converge by dominated convergence. Thus both sides of the stability inequality pass to the limit. Finally, Lemma~\ref{lem:z} gives $v(r)>c_sr^\tau$, hence $u_a(x)>U_q(x)$ for $x\ne0$.
\end{proof}

\begin{proof}[Proof of Theorem~\ref{thm:sharp}(b)]
Let $r=H_0(x)$ and $c:=p^{p-1}(N-p)$. Since $|\nabla U_e|=p|\nabla H_0|/r\le C/r$ and $N>p$, we have $U_e\in W^{1,p}_{\rm loc}(\mathbb{R}^N)$ and $e^{U_e}=cr^{-p}\in L^1_{\rm loc}(\mathbb{R}^N)$. By Lemma~\ref{lem:radial} with $V(r)=\log c-p\log r$, $\mathcal A(\nabla U_e)=-p^{p-1}r^{1-p}x/r$ and $-\Delta_p^HU_e=p^{p-1}r^{1-N}(r^{N-p})'=e^{U_e}$ in $\mathbb{R}^N\setminus\{0\}$. Since $|\mathcal A(\nabla U_e)|\le Cr^{1-p}$, the flux of $\mathcal A(\nabla U_e)$ through $\{H_0=\rho\}$ is $O(\rho^{N-p})\to0$, so $U_e$ is a weak solution in $\mathbb{R}^N$.

If $N<\frac{p(p+3)}{p-1}$, then $U_e$ is not stable by Corollary~\ref{cor:stable}(c). Let $N\ge\frac{p(p+3)}{p-1}$, put $b:=\frac{N-p}2$ and $W:=H_0^{-b}$. Here $\nabla U_e=-pr^{-1}\nabla H_0$ and $\nabla W=-br^{-b-1}\nabla H_0$ are parallel, $H(\nabla U_e)=p/r$, and $\nabla H(\nabla U_e)=-x/r$. Hence $D^2H(\nabla U_e)\nabla W=0$, $\langle\nabla H(\nabla U_e),\nabla W\rangle=br^{-b-1}$, and
\[
\mathcal M(\nabla U_e)\nabla W=-(p-1)p^{p-2}b\,r^{1-p-b}\,\frac xr .
\]
By \eqref{eq:div}, in $\mathbb{R}^N\setminus\{0\}$
\[
-\operatorname{div}\big(\mathcal M(\nabla U_e)\nabla W\big)=(p-1)p^{p-2}b\,r^{1-N}\big(r^{N-p-b}\big)'=(p-1)p^{p-2}\frac{(N-p)^2}4\,r^{-p}W .
\]
The inequality $(p-1)p^{p-2}\frac{(N-p)^2}4\ge p^{p-1}(N-p)$ is equivalent to $(p-1)(N-p)\ge4p$, that is, to $N\ge\frac{p(p+3)}{p-1}$. Thus $-\operatorname{div}(\mathcal M(\nabla U_e)\nabla W)\ge e^{U_e}W$ in $\mathbb{R}^N\setminus\{0\}$, and Lemma~\ref{lem:picone} with $\Omega=\mathbb{R}^N\setminus\{0\}$ gives $Q_{U_e}(\varphi)\ge\int e^{U_e}\varphi^2$ for $\varphi\in C^1_c(\mathbb{R}^N\setminus\{0\})$.

It remains to remove the origin. Let $\varphi\in C^1_c(\mathbb{R}^N)$ and $\varphi_\rho:=\varphi(1-\eta_\rho)$. Since $|\mathcal M(\nabla U_e)|\le Cr^{2-p}$,
\[
\int|\mathcal M(\nabla U_e)|\,\varphi^2|\nabla\eta_\rho|^2\le C\rho^{-2}\int_{B_{2\rho}}|x|^{2-p}\le C\rho^{N-p}\to0,
\]
and $r^{2-p}|\nabla\varphi|^2$ and $e^{U_e}\varphi^2$ are integrable near $0$ because $N>p$. The mixed term in $Q_{U_e}(\varphi_\rho)$ tends to zero by Cauchy--Schwarz for $\mathcal M(\nabla U_e)\ge0$. Hence $Q_{U_e}(\varphi_\rho)\to Q_{U_e}(\varphi)$ and $\int e^{U_e}\varphi_\rho^2\to\int e^{U_e}\varphi^2$, and $U_e$ is stable.
\end{proof}

\begin{proof}[Proof of Theorem~\ref{thm:sharp}(c)]
Let $r=H_0(x)$, $N=p$, and $\beta:=(\tau-1)(p-2)$. Since $\tau(q+p-1)=p$, we have $p-q\tau=\tau(p-1)$ and $\beta-2=-\tau(q+1)$. By Lemma~\ref{lem:radial} with $V(r)=c_sr^\tau$,
\[
\mathcal A(\nabla U_q)=(c_s\tau)^{p-1}r^{(\tau-1)(p-1)}\frac xr,\qquad \Delta_p^HU_q=(c_s\tau)^{p-1}(N-q\tau)r^{-q\tau}=U_q^{-q}
\]
in $\mathbb{R}^N\setminus\{0\}$, where we used $c_s^{q+p-1}\tau^{p-1}(N-q\tau)=1$. Since $N+p(\tau-1)=p\tau>0$ and $q\tau<p=N$, we have $U_q\in W^{1,p}_{\rm loc}(\mathbb{R}^N)$ and $U_q^{-q}\in L^1_{\rm loc}(\mathbb{R}^N)$, and the flux of $\mathcal A(\nabla U_q)$ through $\{H_0=\rho\}$ is $O(\rho^{N-1+(\tau-1)(p-1)})=O(\rho^{\tau(p-1)})\to0$. Hence $U_q$ is a weak solution in $\mathbb{R}^N$, and it is positive away from the origin.

If $q>\frac{(p-2)^2}4$, then $U_q$ is not stable by Corollary~\ref{cor:stable}(b). Let $q\le\frac{(p-2)^2}4$, put $b:=\frac{N-2+\beta}2=\frac{\tau(p-2)}2>0$ and $W:=H_0^{-b}$. Here $H(\nabla U_q)=c_s\tau r^{\tau-1}$, $\nabla H(\nabla U_q)=x/r$, and $\nabla W=-br^{-b-1}\nabla H_0$ is parallel to $\nabla U_q$. Hence $D^2H(\nabla U_q)\nabla W=0$ and $\mathcal M(\nabla U_q)\nabla W=-(p-1)(c_s\tau)^{p-2}b\,r^{\beta-b-1}\frac xr$, and \eqref{eq:div} gives
\[
-\operatorname{div}\big(\mathcal M(\nabla U_q)\nabla W\big)=(p-1)(c_s\tau)^{p-2}b\,(N-2+\beta-b)\,r^{\beta-b-2}=(p-1)(c_s\tau)^{p-2}b^2\,r^{-\tau(q+1)}\,W
\]
in $\mathbb{R}^N\setminus\{0\}$. On the other hand, $qU_q^{-q-1}W=qc_s^{-q-1}r^{-\tau(q+1)}W$, and $c_s^{-q-1}=c_s^{p-2}\tau^{p-1}(p-q\tau)=c_s^{p-2}\tau^p(p-1)$. Thus $-\operatorname{div}(\mathcal M(\nabla U_q)\nabla W)\ge qU_q^{-q-1}W$ if and only if
\[
(p-1)\tau^{p-2}\frac{\tau^2(p-2)^2}4\ge q\tau^p(p-1),\qquad\text{that is,}\qquad q\le\frac{(p-2)^2}4 .
\]
Lemma~\ref{lem:picone} with $\Omega=\mathbb{R}^N\setminus\{0\}$ and $V=qU_q^{-q-1}$ gives the stability inequality for $\varphi\in C^1_c(\mathbb{R}^N\setminus\{0\})$. To remove the origin, let $\varphi\in C^1_c(\mathbb{R}^N)$ and $\varphi_\rho=\varphi(1-\eta_\rho)$. Since $|\mathcal M(\nabla U_q)|\le Cr^\beta$ and $N-2+\beta=\tau(p-2)>0$,
\[
\int|\mathcal M(\nabla U_q)|\varphi^2|\nabla\eta_\rho|^2\le C\rho^{-2}\int_{B_{2\rho}}|x|^\beta\le C\rho^{\tau(p-2)}\to0,
\]
while $|x|^\beta|\nabla\varphi|^2$ and $U_q^{-q-1}\varphi^2\le C|x|^{\beta-2}\varphi^2$ are integrable near $0$. Again, the mixed term tends to zero by Cauchy--Schwarz for $\mathcal M(\nabla U_q)\ge0$. Letting $\rho\to0$, we conclude that $U_q$ is stable.
\end{proof}

\section*{Declarations}

\noindent\textbf{Conflict of interest.} The authors declare that they have no conflict of interest.

\smallskip
\noindent\textbf{Data availability.} Data sharing is not applicable to this article, as no datasets were generated or analyzed during the current study.

\smallskip
\noindent\textbf{Use of AI tools.} AI-assisted tools were used during the preparation of the manuscript. The authors reviewed and edited all AI-assisted output and takes full responsibility for the content of the work.

\end{document}